\documentclass[nonblindrev]{informs3custom} 

\CustomSpacedXI

\AtBeginDocument{%
  \addtolength{\oddsidemargin}{-20pt}%
  \addtolength{\evensidemargin}{-20pt}%
}

\usepackage{multirow}
\usepackage{subfig}
\usepackage{algorithmic,algorithm}
\usepackage{verbatim}
\usepackage{environ}
\usepackage{color}

\NewEnviron{new}{{\color{red}}}

\newcommand{\IFTHENELSE}[3]{%
    \STATE\algorithmicif\ {#1}\ \algorithmicthen\ {#2},\ \algorithmicelse\ {#3}.%
}

\newtheorem{assume}{Assumption}
\newtheorem{prop}{Proposition}
\newtheorem{thm}{Theorem}
\newtheorem{cor}{Corollary}

\floatname{algorithm}{Policy}

\newcommand{\paragraphtitle}[1]{\paragraph{#1}}

\newcommand{\Q}{}
\newcommand{\Qn}{}

\newcommand{\EST}{\text{EST}}
\newcommand{\PDual}{\text{PD}}
\newcommand{\OPT}{\text{OPT}}
\newcommand{\MLE}{\text{mle}}
\newcommand{\lfrac}[2]{{(#1)}/{(#2)}}   
\newcommand{\M}{{\scriptscriptstyle M}}

\usepackage{natbib}
 \bibpunct[, ]{(}{)}{,}{a}{}{,}%
 \def\bibfont{\small}%
\TheoremsNumberedThrough     
\ECRepeatTheorems

\EquationsNumberedThrough    

\MANUSCRIPTNO{}

\begin{document}

\TITLE{Yield Optimization of Display Advertising with Ad Exchange}
\RUNTITLE{Yield Optimization with Ad Exchange}

\ARTICLEAUTHORS{%
\AUTHOR{Santiago Balseiro}
\AFF{Graduate School of Business, Columbia University, New York, NY 10027, \EMAIL{srb2155@columbia.edu}} 
\AUTHOR{Jon Feldman, Vahab Mirrokni, S. Muthukrishnan}
\AFF{Google Research, New York, NY 10011, \EMAIL{jonfeld@google.com}, \EMAIL{mirrokni@google.com}, \EMAIL{muthu@google.com}}
} 

\RUNAUTHOR{Balseiro et al.}

\ABSTRACT{%
\begin{center}\normalsize September 2011\\[1ex] \end{center}

It is clear from the growing role of Ad Exchanges in the real-time sale of advertising slots that web publishers are considering a new alternative to their more traditional reservation-based ad contracts.  To make this choice, the publisher must trade off, in real-time, the short-term revenue from an Ad Exchange with the long-term benefits of delivering good quality spots to the reservation ads.

In this paper, we formalize this combined optimization problem as a stochastic control problem and derive an efficient policy for online ad allocation in settings with general joint distribution over placement quality and exchange prices. We prove asymptotic optimality of this policy in terms of any arbitrary trade-off between quality of delivered reservation ads and revenue from the exchange, and provide a rigorous bound for its convergence rate to the optimal policy.  We also give experimental results on data derived from real publisher inventory, showing that our policy can achieve any Pareto-optimal point on the quality vs. revenue curve.
}
\maketitle

\section{Introduction}

Internet Display Advertising refers generally to the graphical and video ads that are now ubiquitous on the web. These types of ads generated about 10 billion dollars in the US in 2010, and analysts see a clear rising trend~\citep{IAB2011}.  Traditionally, an advertiser would buy display ad placements by negotiating deals directly with a publisher (the owner of the web page), and signing an agreement, called a guaranteed contract. These deals usually take the form of a specific number of ad impressions reserved over a particular time horizon (e.g., one million impressions over a month). A publisher can make many such deals with different advertisers, with potentially sophisticated relationships between the advertisers' targeting criteria. The publisher would then need to assign arriving impressions to the matching reservations so as to maximize the \emph{placement quality} of the contracts. Typically, the probability that a user clicks on an ad (known as {\em click-trough rate}) is used as a metric of placement quality.

Guaranteed contracts can suffer in efficiency: since slots are booked in advance, both parties cannot react to instantaneous changes to traffic patterns or market conditions. However, this has changed in the last couple of years. Advertisers may now purchase ad placements through spot markets for the real-time sale of online ad slots, called Ad Exchanges. Prominent examples of exchanges are Yahoo's RightMedia, Microsoft's AdECN, Google's DoubleClick and OpenX. While exchanges differ in their implementations, in a generic Ad Exchange (AdX)~\citep{Muthu2009}, publishers post an ad slot with a reservation price, advertisers post bids, and an auction is run; this happens between the time a user visits a page and the ad is displayed. Ad exchanges allow advertisers to bid in real time and pay only for valuable customers, instead of bulk buying impressions and targeting large audiences.


In presence of ad exchanges, publishers face the problem of maximizing the overall placement quality of the impressions assigned to the reservations together with the total revenue obtained with AdX, while complying with the contractual obligations. Note that these two objectives are potentially conflicting; in the short-term, the publisher might boost the revenue stream from AdX at the expense of assigning lower quality impressions to the advertisers. In the long term, however, it may be convenient for the publisher to prioritize her advertisers in view of attracting future contracts. So for a given piece of ad inventory, the publisher must quickly decide whether to send the inventory to AdX (and at what price), or to assign it to an advertiser with a reservation.

In this paper, we study the problem faced by the publisher, jointly optimizing over AdX and the reservations.  We bring to bear techniques from revenue management and stochastic optimal control, perform a probabilistic modeling of the problem, and derive an efficient policy for making real-time ad allocation decisions. We prove that our policy is asymptotically optimal in terms of an arbitrary (i.e., publisher-defined) trade-off between quality delivered to reservation ads and revenue from the exchange.  Our policy and analysis is quite general, and works for any joint distribution over placement quality and exchange prices, even allowing correlation between advertisers, or between quality and exchange prices. In particular, we provide a rigorous bound on the convergence rate of our policy to the optimal policy (Theorem~\ref{thm:main}). Typically ad allocation research compares to the optimal offline policy in hindsight; instead, we compare our policy with an optimal online policy, obtaining a bound on additive regret, as in online machine learning.


Since the optimal policy cannot be computed efficiently in most real-world problems, we derive a provably good policy which resembles a bid-price control but extended with a pricing function to take into account for AdX. Our policy assigns each guaranteed contract a bid-price (or dual variable), which may be interpreted as the opportunity cost of assigning one additional impression to the advertiser. When a user arrives, the pricing function quotes a reserve price to submit to the exchange that depends on the opportunity cost of assigning the impression to an advertiser. If AdX price does not exceed this reserve price, the impression is immediately assigned to the advertiser whose placement quality exceeds its opportunity cost by the largest amount. We also give experimental results on data derived from real publisher inventory, showing that our policy can achieve any Pareto-optimal point on the quality vs. revenue curve.

\subsection{Related Work and Contributions}

Our works draws on three streams of literature, namely, that of Display Advertising, Revenue Management, and Online Allocation. Rather than attempting to exhaustively survey the literature on each area, we focus on the work more closely related to ours.

\paragraph{Display Advertising.}
There has been recent work on display ad allocation with both contract-based advertisers and spot market advertisers. \citet{GMPV09} focus on ``fair'' representative bidding strategies in which the publisher bids on behalf of the contract-based advertisers competing with the spot market bidders. This line of work is mainly concerned with computing such fair representative bidding strategies for contract-based advertisers. \cite{Chen2011} considers the case when the publisher runs the exchange, and employing a mechanism design approach he characterizes, through dynamic programming, the optimal dynamic auction for the spot market. In this model both bids from the spot market and the total number of impressions are stochastic. We focus, instead, on combined yield optimization and present a model and an algorithm taking into account any trade-off between quality delivered to reservation ads and revenue from the spot market.

\cite{Yang2010} studied the problem faced by the publisher of allocating between the two markets using multi-objective programming. As in our work, they consider different objectives for the publisher, such as, minimizing the penalty of under-delivery, maximizing the revenue from the spot market and the representativeness of the allocation. However, they employ a deterministic model with no uncertainty in which future inventory and contracts are nodes in a bipartite graph. \cite{Alaei2009} proposed an utility model that accounts for two types of advertisers: one oriented towards campaigns and seeking to create brand equity, and the other oriented towards the spot market and seeking to transform impressions to sales. Here impressions are commodities which can be assigned interchangeably to any advertisers. In this setting they look for offline and online algorithms aiming to maximize the utility of their contracts of the allocation. \cite{RoelsFridgeirsdottir2009} studied the scheduling problem in display advertising in the case without the exchange. In this paper the publisher needs to decide, as new contracts arrive, whether to accept them or not, and then dynamically deliver arriving impressions to them. They take into account uncertainty both in supply and demand, provide a dynamic programming formulation, and propose a certainty-equivalent control.



\paragraphtitle{Revenue Management.} Another stream of relevant work is that of Revenue Management (RM). Even though RM is typically applied to airlines, car rentals, hotels and retailing~\citep{TalluriVanRyzin2004}, our problem formulation and analysis is inspired by RM techniques. As in the prototypical RM problem, we look for a policy maximizing the {\em ex-ante} expected revenue, which can be obtained using dynamic programming (DP). Since the resulting DP is intractable, we aim for a deterministic version in which stochastic quantities are replaced by their expect values and quantities assumed to be continuous. These are common in the literature~\citep{GallegoVanRyzin1994, LiuVanRyzin2008}, and provide policies with provably good performance. Indeed, we show that our policy is asymptotically optimal.

The Display Ad problem can be thought of as a parallel-flight Network RM problem (see, e.g., \cite{TalluriVanRyzin1998}) in which users' click probabilities are requests for itineraries, and advertisers are edges in the network. The are three differences, however, with the traditional Network RM problem. First, we aim to satisfy all contracts, or completely deplete all resources by the end of the horizon. Second, in the traditional problem requests are for only one itinerary (which can be accepted or rejected), while in our model each impression can be potentially assigned to any contract and the publisher needs to decide whom to assign the impression based on possibly correlated placement qualities. Finally, the publishers in display advertising may submit impressions to a spot market to increase their revenues.

\begin{table}
    \centering
    \def\arraystretch{1.4}
	\begin{tabular}{c|c|c}
		& \textbf{\quad Network RM \quad } & \textbf{Display Ads} \\ \hline
		Resources & Legs (edges) & Contracts  \\ 
		Constraints & $\le$ & $ = $ \\ 
		Objective & Fares & Placement Quality \\ 
		Decision & Accept/Reject & Decide whom to assign to \\ 
		Spot market  & No & Yes
	\end{tabular}
    \caption{Comparison of the Display Ad and Network Revenue Management problems.}
\end{table}

A popular method for controlling the sale of inventory in revenue management applications is the use of bid-price controls. These were originally introduced by~\cite{Simpson1989}, and thoroughly analyzed by~\cite{TalluriVanRyzin1998}. In this setting, a bid-price control sets a threshold or bid price for each advertiser, which may be interpreted as the opportunity cost of assigning one additional impression to the advertiser.  This approach is standard in the context of revenue maximization, e.g. the stochastic knapsack problem by~\cite{LeviRadovanovic2010}. From this perspective, our contribution is the inclusion of a spot market, the exchange, as an new sales channel. In this case, our policy is a suitable modification of a bid-price control that takes into account AdX by incorporating a pricing function. This pricing function quotes a reserve price to submit to the exchange that depends on the opportunity cost of assigning the impression to an advertiser.

\begin{new}
Our results can be alternatively interpreted as the publisher bidding on behalf of the guaranteed contracts in a sequence of repeated auctions run by the exchange. Here, users visiting the web-page are a scarce resource and advertisers compete to show an ad in the users' screen. The pricing function and the bid-prices determine a bid for the contracts, which competes with the other bids in the exchange. This bid takes into account the value, in terms of placement quality, perceived by the reservation advertisers together with the capacity of the contracts. In this dual interpretation of the problem the spot market lies in the spotlight, while the guaranteed contracts are pushed to the background. Most publishers, however, aim first to fulfill their reservations and then submit their remnant inventory to the exchange. Thus, given the current state of the industry we believe that the first interpretation is more appealing.
\end{new}

\paragraphtitle{Online Allocation.} Our work is closely related to online ad allocation problems, including the {\em Display Ads Allocation (DA)} problem~\citep{Feldman2009, Feldman2010, AWY09, Vee2010}, and the {\em AdWords (AW)} problem~\citep{Mehta2007, DevenurHayes2009}. In both of these problems, the publisher must assign online impressions to an inventory of ads, optimizing efficiency or revenue of the allocation while respecting pre-specified contracts.

In the DA problem, advertisers demand a maximum number of eligible impressions, and the publisher must allocate impressions that arrive online to them. Each impression has a potentially different value for every advertiser. The goal of the publisher is to assign each impression to one advertiser maximizing the value of all the assigned impressions. The adversarial online DA problem was considered in~\cite{Feldman2009}, which showed that the problem is inapproximable without exploiting {\em free disposal}; using this property (that advertisers are at worst indifferent to receiving more impressions than required by their contract), a simple greedy algorithm is $1\over 2$-competitive, which is optimal. When the demand of each advertiser is large, a $(1-{1\over e})$-competitive algorithm exists~\citep{Feldman2009}, and it is tight. The stochastic model of the DA problem is more related to our problem. Following a training-based dual algorithm by \cite{DevenurHayes2009},  training-based $(1-\epsilon)$-competitive algorithms have been developed for the DA problem and its generalization to various packing linear programs~\citep{Feldman2010, Vee2010, AWY09}. 

In the AW problem, the publisher allocates impressions resulting from search queries. Here each advertiser has a budget on the total spend instead of a bound on the number of impressions. Other than training-based dual algorithms and primal-dual algorithms that get similar bounds as in the DA problem~\citep{DevenurHayes2009}, online adaptive optimization techniques have been applied to online stochastic ad allocation~\citep{srikant}. Such control-based adaptive algorithms achieve asymptotic optimality following an updating rule inspired by the primal-dual algorithms.

Our work differs from all the above in three main aspects: (i) We study both the parametric and non-parametric models, and compare their effectiveness in terms of the size of the sample sizes---both analytically for various distributions and experimentally on real data sets. (ii) Instead of using the framework of competitive analysis and comparing the solution with the optimum solution in hindsight, we compare the performance of our algorithm with the optimal online policy, and present a rate of convergence bound under this model. This is akin to regret bounds found in online Machine Learning; and (iii) None of the above work considers the simultaneous allocation of reservation ads and ads from AdX. In particular, these previous works do not consider the trade-off between the revenue from a spot market based on real-time bidding and the efficiency of reservation-based allocation.


It is tempting to simply reduce to online stochastic packing by considering the AdX as just another ``advertiser.'' The problem with this is that it does not allow adjusting the reserve price, or allocating a reservation advertiser if the AdX rejects.  In fact one can make such a reduction go through by considering online decisions on pairs of (reserve price, advertisers), and formalize the problem as an online allocation problem with general packing constraints. After a couple more steps, one can apply the techniques of ~\cite{DevenurHayes2009,Feldman2010,Vee2010,AWY09} to derive an online algorithm for the combined problem. However this approach discretizes the price space into multiples of $\delta$; thus (a) we lose $1+\delta$ in the yield, (b) we increase the running time by a factor of $1\over \delta$, and (c) for the the competitiveness proof to hold, we need more stringent conditions on the size of the weights. 
In addition, there is no clear way to apply the parametric technique. The method presented in this work not only avoids this dependence, it is a much more natural, extensible solution to the problem.


\section{Model}\label{sec:model}
Consider a publisher displaying ads in a web page. The web page has a single slot for display ads, and each user is shown at most one impression per page. The publisher has signed contracts with a set $A$ of advertisers guaranteeing them a certain number of targeted impressions within a given time horizon. We denote by $\mathcal{A} = \{1,\ldots,A\}$ the set of advertisers.

Even though the number of users visiting a web page is uncertain, publishers usually have fairly good estimates of the total number of expected users that arrive in a given horizon. In this model we index time based on the arrival of each user, and assume that the total number of users is fixed and equal to $N$. We do allow users to have different characteristics (random number of users can be accommodated in our model by considering dummy arrivals). Indeed, depending on the user profile, the impression may be more or less attractive for different advertisers.



We assume that the $n$-th impression is endowed with a vector of placement qualities $Q_n = \{ Q_{n,a} \}_{a \in \mathcal{A}}$, where $Q_{n,a}$ is the predicted quality advertiser $a$ would perceive if the impression is assigned to her. Qualities lie in some compact space $\Omega \subseteq \mathbb{R}^A$. A typical measure of placement quality is the estimated probability that the user click on each ad. In practice, such measure of quality is learned using, for example, logistic regression. Here we abstract from the learning problem and assume that qualities $\{Q_n\}_{n=1,\ldots,N}$ are random and drawn independently from some joint c.d.f $G(\cdot)$. We do allow, however, for qualities to be jointly distributed across advertisers. This captures the fact that advertisers might have similar target criteria, and hence the qualities perceived might be correlated. We do not impose any further restrictions on the qualities, other than finite second moments. Notice that the publisher observes the realization of the placement quality before showing the ad.

The publisher has agreed to deliver exactly $C_a$ impressions to advertiser $a \in \mathcal{A}$; neither over-delivery nor under-delivery is allowed. We denote by $\rho = \{\rho_a\}_{a \in \mathcal{A}}$ with $\rho_a = \frac {C_a} N$, the capacity to impression ratio of each advertiser. Note that a necessary condition for the feasibility of the operation is that the number of arriving impressions suffices to satisfy the contracts, or $\sum_{a \in \mathcal{A}} \rho_a \le 1$. An assumption of this general model is that any user can be potentially assigned to any advertiser. In practice each advertiser may be interested in a particular group of user types. It is important to note that this is not a limitation of our results, but rather a modeling choice; in \S\ref{sec:data-model} we show how to handle targeting criteria by setting $Q_{n,a} = -\tau_a$ for impressions not matching the targeting criteria of an advertiser.  This can also be interpreted as forcing the publisher to pay a good will penalty $\tau_a$ to the advertisers each time an undesired impression is incorrectly assigned.

Arriving impressions may either be assigned to the advertisers, discarded or auctioned in the Ad Exchange (AdX) for profit.  In a general AdX~\citep{Muthu2009}, the publisher contacts the exchange with a minimum price she is willing to take for the slot. Additionally, the publisher may submit some partial information of the user visiting the website. For simplicity, we first assume that no information about the user is revealed. However, in section~\ref{sec:AdX-user-info} we relax this assumption. Internally the exchange contacts different ad networks, and in turn they return bids for the slot. The exchange determines the winning bid among those that exceed the reserve price via an auction, and returns a payment to the publisher. In this case we say that the impressions is \emph{accepted}, and the publisher is contractually obligated to display the winning impression. In the case that no bid attains the reserve price, no payment is made and the impression is \emph{rejected}. We present the formal model of the exchange in Section~\ref{sec:AdX}.  The entire operation above is executed before the page is rendered in the user's screen.
Thus, in the event that the impression is rejected by the exchange, the publisher may still be able to assign it to some advertiser. Figure \ref{fig:decisions} summarizes the decisions involved.

\begin{figure}[t]
\centering
\includegraphics[width=0.6\textwidth]{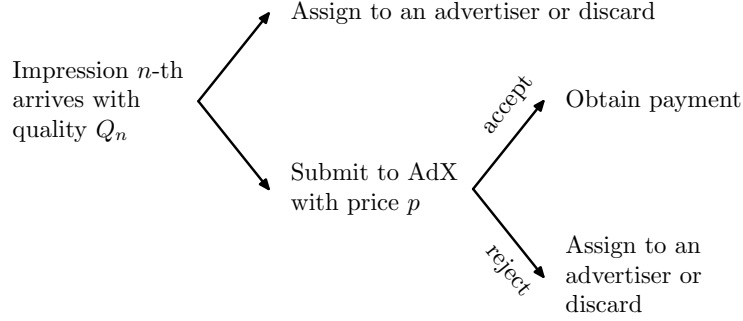}
\caption{Publisher's decision tree for a new impression.}\label{fig:decisions}
\end{figure}

For notational simplicity we extend the set of advertisers to $\mathcal{A}_0 = \{0\} \cup \mathcal{A}$ by including an outside option $0$ that represents discarding an impression.  We set the quality of the outside option identically to zero, i.e. $Q_{n,0} = 0$ for all impressions $n = 1,\ldots,N$. In the following, the terms discarding an impression or assigning it to advertiser $0$ are used interchangeably. We set $\rho_0 = 1 - \sum_{a\in\mathcal{A}}\rho_a$ to be the fraction of impressions that are not assigned to any advertiser. To wit, a fraction of the $\rho_0$ impressions will be assigned to the winning impression of AdX, and the remainder effectively discarded.

\subsection{Objective}\label{sec:obective}
The publisher's problem is to maximize the overall placement quality of the impressions assigned to the advertisers together with the total revenue obtained with AdX, while complying with the contractual obligations. Note that the objectives are potentially conflicting; in the short-term, the publisher might boost the revenue stream from AdX at the expense of assigning lower quality impressions to the advertisers. In the long term, however, it may be convenient for the publisher to prioritize her advertisers, in view of attracting future contracts.

We attack the multi-objective problem by taking a weighted sum of both objectives. The publisher has at her disposal a parameter $\gamma$, which allows her to trade-off between these conflicting objectives. The aggregated objective is given by
$$
    \text{yield} = \text{revenue(AdX)} + \gamma \cdot \text{quality(advertisers)},
$$
Hence, by choosing a suitable large $\gamma$ the advertisers may focus on assigning high quality impressions to the advertisers; while a small $\gamma$ would prioritize the revenue from AdX (the publisher may set different values of the parameter for each advertiser). Without loss of generality, we set $\gamma = 1$ for the remainder of this paper, except when noted otherwise.

\begin{new}
By adjusting the tradeoff parameter $\gamma$ the publisher is able to construct the Pareto efficient frontier of attainable revenue from the AdX and quality for the advertisers. In Section XX we show this curve for real publisher data.
\end{new}

Alternatively, the publisher might impose that the overall quality of the impressions assigned to the advertiser is greater than some threshold, and then maximize the total revenue obtained from AdX; this may have a more natural interpretation for some publishers, and would be simpler than having to set $\gamma$.  We can model this simply by interpreting $\gamma$ as the Lagrange multiplier of the quality of service constraint, and our problem as the Lagrange relaxation of the constrained program. In \S\ref{sec:quality-constraints} we analyze the implications of this formulation, and in \S\ref{sec:impact-adx} we study experimentally the impact of the choice of $\gamma$ on both objectives.

\subsection{AdX Model} \label{sec:AdX}
The publisher submits an impression to AdX with the minimum price
it is willing to take, denoted by $p\ge0$. The impression is accepted
if there is a bid of value $p$ or more. We
denote by $B$ the winning bid random variable. In the following we
assume that bids are independent of the quality of the impression, and
identically distributed according to a c.d.f. $F(\cdot)$. Hence, the
impression is accepted with probability $1-F(p) = \bar{F}(p)$. In this
first model, when the impression is accepted, the publisher is paid
the minimum price $p$. In Sections~\ref{sec:AdX-multiple-bidders} and~\ref{sec:AdX-user-info} we drop this assumption and consider a more general second-price auction with side information.

Suppose the publisher has computed an {\em opportunity cost} $c$ for selling this inventory in the exchange; that is, the publisher stands to gain $c$ if the impression is given to a reservation advertiser.Given opportunity cost $c\ge0$ the publisher picks the price that maximizes its expected revenue. Hence, the publisher solves the optimization problem $R(c) = \max_{p\ge0} \bar{F}(p) p + F(p) c$. Changing variables, we can define $r(s) = s \bar{F}^{-1}(s)$ to be the expected revenue under acceptance probability $s$, and rewrite this as
\begin{align} \label{R(c)}
R(c) &= \max_{s \in [0,1]} r(s) + (1 - s) c.
\end{align}
Also, let $s^*(c)$ be the least maximizer of \eqref{R(c)}, and $p^*(c) = \bar{F}^{-1} \left(s^*(c)\right)$ be the price that verifies the maximum.

\begin{assume}\label{regular} The expected revenue under survival probability $s$ is continuous, concave, non-negative, bounded, and satisfies $\lim_{s \rightarrow 0} r(s) = 0$. We call a function $r(s)$ that satisfies all of the assumptions above a \emph{regular revenue function}.
\end{assume}

These assumptions are common in RM literature (see, e.g., \cite{GallegoVanRyzin1994}). A sufficient condition for the concavity of the revenue is that $B$ has increasing generalized failure rates~\citep{Lariviere2006}. Regularity implies, among other things, the existence of a \emph{null price} $p_\infty$ such that $\lim_{p \rightarrow p_\infty} \bar{F}(p) p = 0$. Additionally, it allows us to characterize the value function $R(c)$. In \S\ref{sec:AdX-multiple-bidders} we show that $r(s)$ remains regular in the presence of multiple bidders in the AdX by considering the joint density of the highest and second-highest bids. Thus, all our results hold in this case too.


\begin{prop} \label{thm:revenue-regularity}
Suppose that $r(s)$ is regular revenue function. Then, $R(c)$ is non-decreasing, convex, continuous, and $R(c) \ge c$. Additionally, $R(c) - c$ is non-increasing, $s^*(c)$ is non-increasing, and $p^*(c)$ is non-decreasing.
\end{prop}

An important consequence of above is that the maximum revenue expected from submitting an impression to AdX is always greater than the opportunity cost. This should not be surprising, since the publisher can pick a price high enough to compensate for the revenue loss of not assigning the impression. Hence, assigning an impression directly to an advertiser (rather than first testing the exchange) is never the right decision, and so in Figure \ref{fig:decisions} the upper branch is never taken.

\begin{new}
Publishers usually receive a revenue share of all impressions sold in the exchange, and no fixed cost is charged for using the exchange (the publisher may still incur a technological setup cost for configuring the system). Under such a \emph{revenue sharing scheme} the exchange keeps a fraction $\alpha$ of the bidder's payment, and the publisher receives the amount $(1-\alpha) p$ for the impression. Our model can accommodate the cost of the exchange by noticing that the publisher only needs to increase the impression's opportunity cost to $c / (1-\alpha)$ to take into account the revenue sharing scheme. It is straightforward to show that, when the exchanges charges a fraction $\alpha$, the publisher's optimization problem is now given by $R_\alpha(c) = (1-\alpha) R(c / (1-\alpha))$, and the optimal price is $p_\alpha^*(c) = p^*(c / (1 - \alpha))$.
\end{new}

\begin{new}
\subsection{Discussion of the Assumptions}

One of our main assumptions is that user attributes are independent across time. Unique user visiting the website arrive essentially at random, so inter-temporal correlation is very weak. Some contracts specify that an ad should not be shown to the same user more than a number given number of times per day. Recurring users can be frequency capped??

It is typically to observe that traffic patterns vary through the day. For example, a newspaper may observe a spike of users connecting in the mornings through their office, an another at time through their home. Our model can accommodate time-dependent traffic patterns in a straight-forward way.

Our model accounts for correlation between the placement quality of the contracts and correlation between the bids from the exchange and the qualities (see section). Thought, targeting in the spot market could be potentially based on similar criteria than the GC, there are several reason that one would expect the correlation no to be perfect. First, publishers usually do not disclose all the information available about the user, and rendering the targeting for the spot market coarser. For example, the user may be registered in the publisher's web-page and disclosed some of her personal information, which may be used by the publisher to improve the targeting of the GC. Second, advertisers in the spot market may perform a behavioral targeting of the users based on \emph{cookies}. Cookies are private bits of information stored in the users computers that allows advertisers to track the past activity of the user on the web. Cookies are dropped by an advertiser when users visit her own web-sites, and are only accessible by her. Thus, a strong component of the spot market bids may have based on private information, resulting in weakly correlated bids.

In section XX we show the impact of correlation.

Another important assumption, is that the publisher is able to adjust the reserve price for each impression. As we shall see one of our main results states that the publisher is better of testing the exchange before assigning the impression to the contracts. Such a results, depends strongly on the assumption that the publisher is able to adjust the reserve to take into account the opportunity cost of ``losing'' an impression of high quality to the exchange. If this is not the case, the publisher would only test the exchange when the expected revenue from AdX exceeds the contracts' opportunity cost. We discuss this further in Section XX.
\end{new}

\section{Problem Formulation}

In this section we start by formulating an optimal control policy for yield maximization based on dynamic programming (DP), where the state of the system is represented by the number of impressions yet to arrive, and a vector of the number of impressions needed to comply with each advertiser's contract. Unfortunately, the state space of the DP has size $O(N^{A+1})$, and in most real-world problems the number of impressions in a single horizon can be in the order of millions. So the DP is not efficiently solvable. We give, instead, an approximation in which stochastic quantities are replaced by their expected values, and are assumed to be continuous. Such ``deterministic approximation problems (DAP)'' are popular in RM (see, e.g.,~\cite{TalluriVanRyzin1998}). In our setting, the approximation we make is to enforce contracts to be satisfied only in expectation. We formulate the problem based on this assumption and obtain an infinite-dimensional program. This DAP is solved by considering its dual problem, which turns out to be a more tractable finite-dimensional convex program. Finally, we wrap a full stochastic policy around it (one that always meets the contracts, not just in expectation).

\subsection{Dynamic Programming Formulation}\label{sec:DP}


Let $(m,X)$ be the state of the system, where we denote by $m$ the total number of impressions remaining to arrive, and by $X = \{x_a\}_{a \in \mathcal{A}}$ the number of impressions needed to comply with each advertiser's contract. Let the value function, denoted by $J_m(X)$, be defined as the optimal expected yield obtainable under state $(m,X)$. Using the fact that is optimal to first test the exchange, we obtain the following Bellman equation
\begin{align}
	J_m(X) &= \mathbb{E}_{Q_n} \Big[ \max_{p\ge0} \Big\{\bar{F}(p)(p+J_{m-1}(X)) + (1-\bar{F}(p)) \max_{a \in \mathcal{A}_0} \{Q_{n,a} + J_{m-1}(X-\mathbf{1}_a)\}  \Big\} \Big] \nonumber\\
	&= J_{m-1}(X) +  \mathbb{E}_{Q_n} \left[ R\left(\max_{a \in \mathcal{A}_0} \{Q_{n,a} - \Delta_a J_{m-1}(X) \}\right)\right] \label{DP},
\end{align}
where we defined $\mathbf{1}_a$ as a vector with a one in entry $a$ and zero elsewhere, $\mathbf{1}_0 = 0$, and $\Delta_a J_m(X) = J_m(X) - J_m(x-\mathbf{1}_a)$ as the expected marginal yield of one extra impression for advertiser $a$. In \eqref{DP} the objective accounts for the yield obtained from attempting to send the impression to AdX. In the yield has two terms that depend on whether the impression is accepted or not by AdX. In the latter case the maximum accounts for the decision of assigning the impression directly to the advertiser or discarding the impression (when $a=0$). In~\eqref{DP} we used the fact that assigning an impression directly to an advertiser is never the right decision (except in boundary conditions, see below). The publisher, however, may choose to discard impressions with low quality after being rejected by AdX.

Our objective is to compute $J_N^* = J_N(C)$. Let $M$ be an upper-bound on the expected yield.\footnote{One could set, e.g., $M \triangleq \max\{ p_\infty, \bar Q \}$ where $\bar Q$ is an upper-bound on the placement quality} The boundary conditions are
\begin{align*}
	J_m(x) &= -M, \qquad \forall \text{$X$ s.t. $x_a<0$ for some $a \in \mathcal{A}$}, \\
	J_m(x) &= -M, \qquad \forall m<\sum_{a \in \mathcal{A}} x_a.
\end{align*}
Recall that when the contract with an advertiser is fulfilled, no extra yield is obtained from assigning to her more impressions. This is the case of the first boundary condition, which guarantees that advertisers whose contract is  fulfilled are excluded from the assignment. In particular, when $X = 0$ all remaining impressions are sent to AdX with the yield maximum price $p^*(0)$ when $x=0$. The second boundary condition guarantees that the contracts with the advertisers are always fulfilled. When $\sum_{a \in \mathcal{A}} x_a = m$ AdX must be bypassed, and impressions should be assigned directly to the advertisers. 
The optimal policy is described in Policy~\ref{policy:dp}.


\begin{algorithm}
\begin{algorithmic}
\STATE Observe state $(m,X)$ and the realization $Q_n$.
\STATE Let $a_n^* = \arg\max_{a \in \mathcal{A}_0} \left\{Q_{n,a} - \Delta_a J_{n-1}(X)\right\}$.
\STATE Submit to AdX with price $p^*\left( Q_{n,a_n^*} - \Delta_{a_n^*} J_{n-1}(X) \right)$.
\IF{impression rejected by AdX and $a_n^*\neq0$}
	\STATE Assign to advertiser $a_n^*$.
\ENDIF
\end{algorithmic}
\caption{Optimal dynamic programming policy.}
\label{policy:dp}
\end{algorithm}

In the above policy, when the impression is submitted to AdX, the optimal price ponders an opportunity cost of $Q_{n,a_n^*} - \Delta_{a_n^*} J_{n-1}(X)$. This opportunity cost, when positive, is just the value of the impression adjusted by the loss of potential yield from
assigning the impression right now.  Note that the two boundary conditions are implicit in the optimal policy. 
This guarantees that the policy complies with the contracts. It is routine to check that the value function $J_n(X)$ is finite for all feasible states and that Policy~\ref{policy:dp} is optimal for the dynamic program in \eqref{DP}. It is worth noting that in order to implement the optimal policy one needs to pre-compute the value function, which is intractable in most real instances.

\subsection{Deterministic Approximation Problem (DAP)}
\label{sec:DAP}

We aim for an approximation in which (i) the policy is independent of the history but dependent on the realization of $Q_n$, (ii) capacity constraints are met in expectation, and (iii) controls are allowed to randomize. These approximations turn out to be reasonable when the number of impressions is large. When an impression arrives, the publisher controls the reserve price submitted to AdX, and the advertiser to whom the impression is assigned, if rejected by AdX. Alternatively, in this formulation we state the controls in terms of total probabilities, where each control is a function from the quality domain to $[0,1]$. Let $\vec{s} = \{s_n(\cdot)\}_{n=1,\ldots,N}$ and $\vec{\imath}=\{i_n(\cdot)\}_{n=1,\ldots,N}$ 
be vectors of functions from $\Omega$ to $\mathbb{R}$, such that when the $n^\text{th}$ impression arrives with quality $Q$ the impression is accepted by AdX with probability $s_n(Q)$, and with probability $i_{n,a}(Q)$ it is assigned to advertiser $a$. The conditional probability of an impression being assigned to advertiser $a$ given that it has been rejected by AdX is given by $I_{n,a}(Q) = i_{n,a}(Q) / (1 - s_n(Q))$. When it is clear from the context, we simplify notation by eliminating the dependence on $Q$ from the controls.


A control is feasible for the DAP if (i) it satisfies the contractual constraint in expectation, (ii) the individual controls are non-negative, and (iii) for every realization of the qualities the probabilities sum up to at most one. We denote by $\mathcal{P}$ the set of controls that satisfy the latter two conditions. That is, $\mathcal{P} = \{(s,i): \sum_{a \in \mathcal{A}} i_{a}\Q + s\Q \le 1,s\ge 0, i \ge 0\}$. The objective of the DAP is to find a sequence of real-valued measurable functions that maximize the expected yield, or equivalently
\begin{subequations} \label{DAP}
\begin{align}
	J_N^D = \max_{(s_n,i_n) \in \mathcal{P}} & \;  \sum_{n=1}^N \mathbb{E} \left[ r(s_n\Qn) + \sum_{a \in \mathcal{A}} i_{n,a}\Qn Q_{n,a}\right] \nonumber\\
	\text{s.t. } & \sum_{n=1}^N \mathbb{E} \left[i_{n,a}\Qn \right] = N \rho_a, \qquad \forall a \in \mathcal{A}.
 \label{DAP:capacity}
\end{align}
\end{subequations}
The first term of the objective accounts for the revenue from AdX, while the second accounts for the quality perceived by the advertisers. Notice that in the DAP we wrote the total capacity as $N \rho_a$ instead of $C_a$ to allow the problem to be scaled.

Alas, the problem is still hard to solve since the number of functions is linear in $N$. However, exploiting the regularity of the revenue function, we can show that in the optimal solution to DAP, we can drop the dependence on $n$ in the controls. This follows from the linearity of the constraints together with the concavity of the objective. We formalize this discussion in the following proposition.

\begin{prop}\label{homocontrols} Suppose that the revenue function is regular. Then, there exists a time-homogenous optimal solution to the DAP, i.e. where $s_n(Q_n) = s(Q_n)$ for all $n=1,\ldots,N$ and $i_n(Q_n) = i(Q_n)$ for all $n=1,\ldots,N$.
\end{prop}

The previous proposition allows us scale the problem so that $N=1$, and consider the maximum expected revenue of one impression, denoted by $J_1^D$. The total revenue for the whole time horizon is then $J_N^D = N J_1^D$. In order to compute the DAP's optimal solution, we consider its dual problem, which we informally derive next.

\paragraphtitle{Derivation of the Dual to DAP.}
\label{sec:dual-derive}
To find the dual, we introduce Lagrange multipliers $v = \{v_a\}_{a \in \mathcal{A}}$ for the capacity constraints \eqref{DAP:capacity}. The Lagrangian, denoted by $\mathcal{L}(s,i;v)$ is
\begin{align*}
	\mathcal{L}(s,i;v) &= \mathbb{E} \left[ r(s\Q) + \sum_{a \in \mathcal{A}} i_a\Q Q_{a}
     - \sum_{a \in \mathcal{A}}v_a \left( i_a\Q - \rho_a \right) \right].
\end{align*}

The dual function, denoted by $\psi(v)$, is the supremum of the Lagrangian over the set $\mathcal{P}$. Thus, we have that
\begin{align*}
	\psi(v) &= \sup_{(s,i) \in \mathcal{P}} \mathcal{L}(s,i;v) \\
	&= \sup_{s \ge 0} \left\{ \mathbb{E} \left[ r(s\Q) \right] +
       \sup_{i \ge 0, \sum_{a \in \mathcal{A}} i_a \le 1 - s} \mathbb{E} \left[ \sum_{a \in \mathcal{A}} i_a\Q (Q_{a} - v_a) \right] \right\} + \sum_{a \in \mathcal{A}}v_a \rho_a \\
	&= \sup_{s \ge 0} \mathbb{E} \left[ r(s\Q) + (1-s) \max_{a \in \mathcal{A}_0} \{Q_a - v_a\} \right] + \sum_{a \in \mathcal{A}}v_a \rho_a \\
	&=  \mathbb{E} R\big(\max_{a \in \mathcal{A}_0} \{Q_a - v_a\}\big) + \sum_{a \in \mathcal{A}}v_a \rho_a
\end{align*}
where the first equation follows from partitioning the optimization between the AdX acceptance and the assignment probability controls, the second from optimizing over the advertiser assignment controls $i$, and the last equation from solving the AdX variational problem. Note that $R$ is convex and non-decreasing and the maximum is convex w.r.t $v$, hence the composite function within the expectation is convex. Using the fact that expectation preserves convexity, we obtain that the objective $\psi(v)$ is convex in $v$.

Next, the dual problem is $\min_{v} \psi(v)$. When the revenue function is regular, the DAP's objective is concave and bounded from above. Moreover, the constraints of the primal problem are linear, and the feasible set $\mathcal{P}$ convex. Hence, by the Strong Duality Theorem (p.224 in~\cite{Luenberger1969}) the dual problem attains the primal objective value. So, we have that dual problem is given by the following convex stochastic problem
\begin{align} \label{DUAL2}
    J_1^D = \min_{v} & \bigg \{ \mathbb{E} R\big(\max_{a \in \mathcal{A}_0} \{Q_a - v_a\}\big) + \sum_{a \in \mathcal{A}}v_a \rho_a \bigg \}.
\end{align}

\paragraphtitle{Deterministic optimal control.}
When the distribution $Q$ is known, the dual problem in \eqref{DUAL2} can be solved using a Subgradient Descent Method. It is worth noting that in many applications the distribution of $Q$ is unknown and should be learned as impressions arrive. We postpone the discussion of that problem until \S\ref{sec:PD-compare}.

Once the optimal dual variables $v$ are known, the primal solution can be constructed from plugging the optimal Lagrange multipliers in $\mathcal{L}(s,i;v)$. Following the derivation of the dual, we obtain that the optimal survival probability is $s(Q) = s^*\left(\max_{a \in \mathcal{A}_0} \{Q_a - v_a\}\right)$. Hence, the impression has a value of $\max_{a \in \mathcal{A}_0} \{Q_a - v_a\}$ for the publisher, and she picks the reserve price that maximizes her revenue. From the optimization over the assignment controls, we see that an impression is assigned to an advertiser $a$ only if she maximizes the contract adjusted quality $Q_a - v_a$. Thus the dual variables $v_a$ act as the \emph{bid-prices} of the guaranteed contracts. Additionally, the impression can be discarded only if the maximum is not verified by an advertiser (i.e. all contract adjusted qualities are non-positive).

Notice that optimizing the Lagrangian states that the impression should be assigned to an advertiser maximizing the contract adjusted quality, but does not specify how the impression should be assigned when --multiple-- advertisers attain the maximum. In the case when the probability of a tie occurring is zero, the problem admits a simple solution: assign the impression to the unique maximizer of $Q_a - v_a$. We formalize this discussion in the the following theorem.

\begin{thm}\label{thm:DAPcontrol} Suppose that the revenue function is regular, and there is zero probability of a tie occurring, i.e. $\mathbb{P} \{Q_a - v_a = Q_a^\prime - v_a^\prime\} = 0$ for all distinct $a,a^\prime \in \mathcal{A}_0$. Then, the optimal controls for the DAP are $s(Q) = s^*\left( \max_{a \in \mathcal{A}_0} \{Q_a - v_a\} \right)$, and $I_a(Q)=\mathbf{1} \left\{Q_a - v_a > Q_{a^\prime} - v_{a^\prime} \: \forall a^\prime \in \mathcal{A}_0  \right\}$, that is, the impression is assigned to the unique advertiser maximizing the contract adjusted quality. Furthermore, the optimal dual variables solve the equations
\begin{align*}
    \mathbb{E} \left[ ( 1 - s^*(Q_a - v_a)) I_a(Q) \right] = \rho_a, \qquad \forall a \in \mathcal{A}.
\end{align*}
\end{thm}

\subsection{Our Stochastic Policy}
\label{sec:policy}
The solution of the DAP suggests a policy for the stochastic
control problem, but we must deal with two technical issues: (i) when
more than one advertiser maximizes $Q_a - v_a$ we need to decide how
to break the tie, and (ii) we are only guaranteed to meet the contracts
in expectation, whereas we must meet them exactly. We defer the first issue until \S\ref{sec:ties}, where we give an algorithm for generalizing the controls $I_a(Q)$ to the case where ties are possible.

We propose a static bid-price control extended with a pricing function for AdX given by $p^*$. The policy, which we denote by $\mu^B$, is defined in Policy~\ref{policy:bp}. In there we let $x_{n,a}$ be the total number of impressions left to assign to advertiser $a$ to comply with the contract, $m=N-n$ the total number of impressions remaining to arrive, and $v$ to be the optimal solution of \eqref{DUAL2}.


\begin{algorithm}
\begin{algorithmic}
\STATE Observe state $(m,X)$ and the realization $Q_n$.
\STATE Let $\mathcal{A}_n = \{a \in \mathcal{A} \; : \; x_{n,a} > 0\}$ be the set of ads yet to be satisfied.
\STATE Let $a_n^* = \arg\max_{a \in \mathcal{A}_n \cup \{0\} } \left\{Q_{n,a} - v_a\right\}$.
\IFTHENELSE{$\sum_{a \in \mathcal{A}} x_{n,a} < m$}{let $p_n = p^*(Q_{n,a_n^*} - v_{a_n^*})$}{let $p_n = p_\infty$}
\STATE Submit to AdX with price $p_n$.
\IF{impression rejected by AdX and $a_n^*\neq0$}
	\STATE Assign to advertiser $a_n^*$.
\ENDIF
\end{algorithmic}
\caption{Static Bid-Price Policy with Pricing $\mu^B$.}
\label{policy:bp}
\end{algorithm}

Notice that impressions are only assigned to advertisers with contracts that have yet to be fulfilled. When all contracts are fulfilled, impressions are sent to AdX with the revenue maximizing price $p^*(0)$. Moreover, when the total number of impressions left is equal to the number of impressions necessary to fulfill the contracts, the price is set to $p_\infty$, and thus all incoming impressions are directly assigned to advertisers. Hence, the stochastic policy $\mu^B$ satisfies the contracts with probability $1$.

The proposed stochastic policy shares some resemblance with the optimal dynamic programming policy. The intuition is that, when the number of impressions is large, the actual state of the system becomes irrelevant because $\Delta_a J_{m-1}(x)$ is approximately constant (for states in likely trajectories), and equal to $v_a$. In that case both policies are equivalent.

\subsection{Handling ties}
\label{sec:ties}

Theorem~\ref{thm:DAPcontrol} had an assumption that there would be no ties between advertisers verifying the maximum $Q_a - v_a$.
In this section we show how to construct a primal optimal solution to the DAP and the corresponding stochastic policy in the general case (for example, when the distribution of placement quality is discrete or has atoms). \cite{DevenurHayes2009} proposed introducing small random and independent perturbations to the qualities, or smoothing the dual problem to break ties. We provide an alternate method that directly attacks ties, and provides a randomized tie-breaking rule.
Computing the parameters of the tie-breaking rule requires solving a flow problem on a graph of size $2^{|A|}$; thus in some settings it may not be possible. In section \ref{sec:computation}, we show that in practice ties do not occur frequently. However, for completeness we provide a full characterization of the problem.

For any non-empty  subset $S \subseteq \mathcal{A}_0$, we define a \emph{$S$-tie} as the event when the maximum is verified exactly by all the advertisers $a \in S$, and the impression is rejected by AdX. Note that the tie may be a singleton, in the case that exactly one advertiser verifies the maximum. Since the dual variables $v$ are known, the probability of such event can be written as
\begin{align*}
	\mathbb{P} (S\text{-tie}) &= \mathbb{E} \Bigl[ \bigl(1 - s^* (\lambda(Q) \bigr) \mathbf{1} \bigl\{ Q_a - v_a = \lambda(Q) \; \forall a \in S, \; Q_a - v_a < \lambda(Q) \; \forall a \notin S \bigr\} \Bigr],
\end{align*}
where $\lambda(Q) = \max_{a \in \mathcal{A}_0} \{Q_a - v_a\}$. With some abuse of notation we define the \emph{$\emptyset$-tie} as the event when the impression is accepted by AdX, that is, $\mathbb{P} (\emptyset\text{-tie}) = \mathbb{E} [s^* (\lambda(Q))]$. Note that the tie events induce a partition of the quality space, and we have that $\sum_{S \subseteq \mathcal{A}_0} \mathbb{P} (S\text{-tie}) = 1$.

We look for a random tie-breaking rule that assigns an arriving impression to advertiser $a \in S$ with conditional probability $I_a(S)$ given that a $S$-tie occurs. Hence, the routing probabilities depend on which advertisers tie, and not on the particular realization of the qualities (they are independent of $\lambda(Q)$). Therefore, under such policy the total probability, originating from $S$-ties, of an impression being assigned to advertiser $a$ is $y_a(S) = \mathbb{P} (S\text{-tie}) I_a(S)$. We can interpret $y_a(S)$ as the normalized flow of impression assigned to the advertiser originating from $S$-ties. We will show that, in terms of $y_a(S)$ as decision variables, finding the tie-breaking rule amounts to solving a transportation problem.

First, in order for the publisher to fulfill the contract with an advertiser $a \in \mathcal{A}$ the incoming flow of impressions over all possible ties sums up to $\rho_a$. The previous constraint can be written as
\begin{align} \label{TRANS:capacity}
	\sum_{S \subseteq \mathcal{A}_0 : a \in S} y_a(S) = \rho_a, 	\qquad \forall a \in \mathcal{A}.
\end{align}
Notice that we impose no constraints for $a=0$ since any number of impressions can be discarded. Alternatively, we could set $\rho_0^\text{eff} = 1 - \mathbb{P} (\emptyset\text{-tie}) - \sum_{a \in \mathcal{A}} \rho_a$ because the impressions effectively discarded are those that are rejected by AdX and not assigned to an advertiser. Second, the outgoing flow of impressions originating from a particular tie should sum up to the actual probability of that tie occurring. Then, we have that
\begin{align} \label{TRANS:prob}
	\sum_{a \in S} y_a(S) = \mathbb{P}(S\text{-tie}), 	\qquad \forall S \subseteq \mathcal{A}_0.
\end{align}
Third, we require that $y_a(S) \ge 0$ for all $S \subseteq \mathcal{A}_0$ and $a \in S$. Finally, in order to obtain the tie-breaking rule we need a non-negative flow satisfying constraints \eqref{TRANS:capacity} and \eqref{TRANS:prob}. Once a such solution is found, the optimal controls can be computed as
\begin{align*}
	I_a(Q) &= \begin{cases}
		y_a(S) / \mathbb{P} (S\text{-tie}) &  \text{if $a \in S$, and $Q$ is an $S$-tie},\\
		0 & \text{otherwise},
		\end{cases}
\end{align*}
with the pricing function as before.

It is not hard to see that the previous problem can be stated as a feasible flow problem in a bipartite graph. We briefly describe how to construct such graph next. On the left-hand side of the graph we include one node for each non-empty subset $S \subseteq \mathcal{A}_0$, and in the right-hand side we add one node for each advertiser $a \in \mathcal{A}_0$. In the following we refer to nodes in the left-hand side as \emph{subset nodes}, and to those in the right-hand side as \emph{advertiser nodes}. The supply for subset nodes is $\mathbb{P}(S\text{-tie})$, while the demand for advertiser nodes is $\rho_a$. Arcs in the graph represent the membership relation, i.e., the subset node $S$ and advertiser node $a$ are connected if and only if $a \in S$. Moreover, arc capacities are set to infinity. In Figure \ref{fig:TRANS} the resulting bipartite graph is shown.

\begin{figure}[t]
\centering
\includegraphics{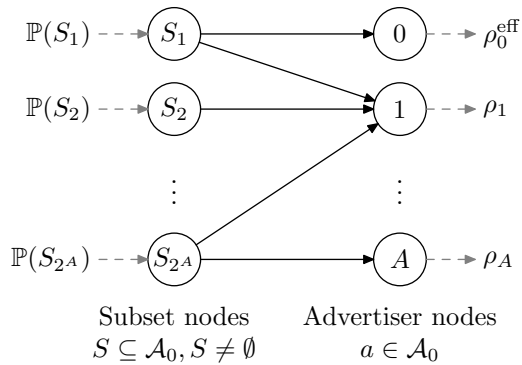}
\caption{Bipartite flow problem solved to obtain the tie-breaking rule. On the left-hand side of the graph we include one node for each non-empty subset $S \subseteq \mathcal{A}_0$ (subset nodes), and in the right-hand side we add one node for each advertiser $a \in \mathcal{A}_0$ (advertiser nodes). The supply for subset nodes is $\mathbb{P}(S\text{-tie})$, while the demand for advertiser nodes is $\rho_a$. Arcs in the graph represent the membership relation, i.e., the subset node $S$ and advertiser node $a$ are connected if and only if $a \in S$. Arc capacities are set to infinity.} \label{fig:TRANS}
\end{figure}

An important question is whether the flow problem admits a feasible solution. The next result proves that the answer is affirmative when the dual variables $v$ are optimal for the dual problem \eqref{DUAL2}. The proof proceeds by casting the feasible flow problem as a maximum flow problem, and then exploiting the optimality conditions of $v$ to lower bound every cut in the bipartite graph.

\begin{prop} \label{prop:flow-feasible} Suppose that $v \in \mathbb{R}^A$ is an optimal solution for the dual problem \eqref{DUAL2}. Then, there exists a non-negative flow satisfying constraints \eqref{TRANS:capacity} and \eqref{TRANS:prob}.
\end{prop}

We conclude this section by showing that the solution constructed is optimal for the primal problem. Notice that the solution is feasible because it satisfies constraints \eqref{TRANS:capacity} and \eqref{TRANS:prob}. In order to prove optimality it suffices to show that it attains the dual objective value, or that it satisfies the complementary slackness conditions. The latter follows trivially.

Once the optimal controls are calculated, we construct our stochastic policy as follows. We let $\mathcal{A}_n^* = \arg\max_{a \in \mathcal{A}_n \cup \{0\} } \left\{Q_{n,a} - v_a\right\}$, be the set of advertisers that attain the maximum. Now, if the impression is rejected by AdX and $\mathcal{A}_n^* \neq \{0\}$, we assign it to advertiser $a$ in $\mathcal{A}_n^*$ with probability $I_a(Q_n) / \sum_{a^\prime \in \mathcal{A}_n^*} I_{a^\prime}(Q_n) $. Notice that impressions are only assigned to advertisers with contracts that have yet to be fulfilled. Additionally, as the contracts of some advertisers are fulfilled, these are excluded of the assignment, and the routing probabilities $I_a(\cdot)$ of the remaining advertisers are scaled-up and normalized.

%

\section{Asymptotic Analysis}

In this section we show that the heuristic policy constructed from the DAP is asymptotically optimal for the stochastic problem when the number of impressions and capacity are scaled up proportionally. We proceed in the following way. First, we formulate the problem as a stochastic control problem (SCP). Though not practical, this abstract and equivalent formulation is useful from a theoretical point of view. Second, we show that the optimal objective value of the DAP provides an upper bound on the objective value of the SCP. Finally, we show that the upper bound is asymptotically tight.

\paragraphtitle{Stochastic Control Problem.} A stochastic control policy maps states of the system to control actions (prices and target advertiser), and is adapted to the history up to the decision epoch. We restrict our attention to policies that always submit the impression to AdX, which were argued to be optimal. Recall that given the reserve price, the publisher knows the actual probability that the impression is accepted by AdX. As before, we recast the problem in terms of the survival probability control. Hence, the publisher picks the probability that the impression is accepted. Conversely, given a survival probability the reserve price can be easily computed using $\bar{F}^{-1}(\cdot)$. We denote by $s_n^\mu(Q) \in [0,1]$ the target survival probability under policy $\mu$ at time $n$ when an impression with quality $Q$ arrives. Similarly, we let $I_{n,a}^\mu(Q) \in \{0,1\}$ indicate whether the $n^{\text{th}}$ impressions is assigned to advertiser $a$ or not when policy $\mu$ is used. In particular, $I_{n,a}^\mu(Q)=1$ indicates that the impression should be assigned to the advertiser if rejected by AdX.

We let the binary random variable $X_n(s_n^\mu\Q)$ indicate whether the $n^{\text{th}}$ impression is accepted by AdX or not when policy $\mu$ is used. Specifically, $X_n(s_n^\mu\Q)=1$ indicates that the impression is accepted by AdX, and when $X_n(s_n^\mu\Q)=0$ the impression is rejected by AdX. Notice that, conditioning on the quality of the impression and the history, $X_n(s_n^\mu\Q)$ is a Bernoulli random variable with success probability $s_n^\mu\Q$.

We denote by $\mathcal{M}$ the set of admissible policies, i.e. policies that are non-anticipating, adapting and feasible. A feasible policy should satisfy the contractual obligations with each advertiser, or equivalently 
$\sum_{n=1}^N \left[1-X_n(s_n^\mu\Qn)\right] I_{n,a}^\mu\Qn = C_a$ in an almost sure sense. Additionally, the target advertiser controls should satisfy that $\sum_{a \in \mathcal{A}} I_{n,a}^\mu\Qn \le 1$, since the impression should be assigned to at most one advertiser. Finally, the equivalent stochastic optimal control problem is
\begin{align}\label{SP}
	J_N^* = \max_{\mu \in \mathcal{M}} & \; \mathbb{E} \left[ \sum_{n=1}^N r(s_n^\mu\Qn) + \left(1-s_n^\mu\Qn\right) \sum_{a \in \mathcal{A}} I_{n,a}^\mu\Qn Q_{n,a} \right], 
\end{align}
where $J_N^*$ denotes the optimal expected revenue over the set of admissible policies $\mathcal{M}$. The objective follows from conditioning on the quality of the impression and the history.
By the Principle of Optimality it is the case that the dynamic program described in section~\ref{sec:DP} provides an optimal solution to the SCP~\citep{Bertsekas2000} and $J_N(C) = J_N^*$.

\paragraphtitle{Analysis.}
Following a similar analysis to \cite{GallegoVanRyzin1994, TalluriVanRyzin1998,LiuVanRyzin2008}, we first show that the optimal objective value of the DAP provides an upper bound to the objective value SCP, and then prove that this bound is tight. For the first result, we proceed by taking the optimal stochastic control policy, and construct a feasible solution for the DAP by taking expectations over the history. Later, we exploit the concavity of the objective and apply Jensen's inequality to show that this new solution attains a greater revenue in the DAP.

\begin{prop} \label{thm:upperbound}
The optimal objective value of the DAP provides an upper bound on the objective value of the optimal policy, i.e. $J_N^* \le J_N^D$.
\end{prop}

Now we complete the analysis by lower bounding the yield of the stochastic policy in terms of the DAP objective. In proving that bound, we look at $N^*$, the first time that any advertisers contract is fulfilled or the point is reached where all arriving impressions need to be assigned to the advertisers. We refer to the time after $N^*$ as the \emph{left-over regime}. The first key observation in the proof is that before time $N^*$, the controls of the stochastic policy behave exactly as the optimal deterministic controls. The second key observation is that the expected number of impressions in the left-over regime is $O(\sqrt N)$, and the left-over regime has a small impact on the objective.

\begin{thm}
\label{thm:main}
Let $J_N^{B}$ be the expected yield under the stochastic policy $\mu^B$. Then,
$$
\frac {J_N^{B}} {J_N^*} \ge \frac {J_N^{B}} {J_N^D} \ge 1 - \frac 1 {\sqrt N} K(\rho),
$$
where
$K(\rho) = \sqrt{ \frac A {A+1} \sum_{a \in \mathcal{A}_0} \frac {1 - \rho_a} {\rho_a}}$.
\end{thm}

\begin{proof}{Proof.}
The first bound follows from Proposition \ref{thm:upperbound}. We now prove the second bound. 

Let $S_{n,a}^\mu = \sum_{i=1}^n \left(1 - X_i(s_i^\mu(Q_i)) \right) I_{i,a}^\mu(Q_i)$ be the total number of impressions assigned to advertiser $a$ by time $n$ when following the stochastic policy $\mu^B$. Additionally, we denote by $S_{n}^\mu = \{ S_{n,a}^\mu \}_{a \in \mathcal{A}}$ the random vector of impressions assigned to advertisers. Then, $x_{n,a}=C_a - S_{n,a}^\mu$ is the total number of impressions left to assign to advertiser $a$ to fulfill the contract, and $m=N-n$ is the total number of impressions remaining to arrive.

To simplify the proof, we let $C_0 = N - \sum_{a\in\mathcal{A}} C_a$ be the total number of impressions that are not assigned to any advertiser (accepted by AdX and discarded), and we refer to $S_{n,0}^\mu = n - \sum_{a\in\mathcal{A}} S_{n,a}$ as total number of impressions not assigned to any advertiser by time $n$ when following the stochastic policy $\mu^B$. Because $C_0$ is the total number of impressions we can dispense of, when the point is reached that $S_{n,0} = C_0$, then all remaining impressions need to be assigned to the advertisers.

Let the random time $N^* = \inf \left\{1\le n \le N \; : \; \text{ $x_{n,a} = 0$ for some $a \in \mathcal{A}$ or $\sum_{a \in \mathcal{A}} x_{n,a} = m$}\right\}$ be the first time that any advertiser's contract is fulfilled or the point is reached where all arriving impressions need to be assigned to the advertisers. Clearly, $N^*$ is a stopping time with respect to the stochastic process $\{ S_{n}^\mu \}_{n=1,\ldots,N}$.

In the following, let $R_n^\mu$ be the revenue from time $n$ under policy $\mu^B$. Similarly, we denote by $R_n$ the revenue from time $n$ when the deterministic control are used in an alternate system with no capacity constraints. Because the deterministic controls are time-homogeneous, and the underlying random variables are i.i.d., then the random variables $\{R_n\}_{n=1,\ldots,N}$ are i.i.d. too. Moreover, it is the case that $\mathbb{E} R_n = J_1^D$. Notice that when $n < N^*$, the controls of stochastic policy $\mu^B$ behave exactly as the optimal deterministic controls. Thus, $R_n = R_n^\mu$ for $n < N^*$. Using this fact together with the fact that $N^*$ is a stopping time we get that
\begin{align}\label{boundOnV}
	J_N^{B} &= \mathbb{E} \left[\sum_{n=1}^N R_n^{\mu} \right] = \mathbb{E} \left[ \sum_{n=1}^{N^*} R_n + \sum_{n=N^*+1}^{N} R_n^{\mu} \right] \ge  \mathbb{E} \left[ \sum_{n=1}^{N^*} R_n \right] = \mathbb{E} N^* J_1^D,
\end{align}
where the inequality follows from the non-negativity of the revenues, and the last equality from Wald's equation. Then, we conclude that $J_N^{B} / J_N^D \ge \mathbb{E} N^* / N$.

Next, we turn to the problem of lower bounding $\mathbb{E} N^*$. Before proceeding we make some definitions. We define by $S_{n,a}$ the number of impressions assigned to advertiser $a$ by time $n$ when following the deterministic controls in the alternate system with no capacity constraints. As for the revenues, it is the case that $S_{n,a} = S^\mu_{n,a}$ for $n < N^*$. We define $S_{n,0}$ in a similar fashion.

Let $N_a = \inf \left\{n \ge 1\; : \; S_{n,a} = C_a \right\}$ be the time when the contract of advertiser $a \in \mathcal{A}$ is fulfilled, and $N_0 = \inf \left\{n \ge 1 \; : \; S_{n,0} = C_0 \right\}$ be the point in time where all arriving impressions need to be assigned to the advertisers. Even though these stopping times are defined with respect to the stochastic process that follows the deterministic controls, it is the case that $N^* = \min_{a \in \mathcal{A}_0} \{ N_a\}$. In the remainder of the proof we study the mean and variance of each stopping time, and then conclude with a bound for $\mathbb{E} N^*$ based on those central moments.

For the case of $a \in \mathcal{A}$, the summands of $S_{n,a}$ are independent Bernoulli random variables with success probability $\rho_a$. The success probability follows from \eqref{DAP:capacity}. Hence, $N_a$ is a negative binomial random variable with $C_a$ successes and success probability $\rho_a$. The mean and variance are given by $\mathbb{E} N_a = N$, and $\text{Var} [N_a] = N \frac {1 - \rho_a} {\rho_a}$, where we used that $\rho_a = C_a / N$. Similarly, for the case of $a=0$, now the summands of $S_{n,0}$ are Bernoulli random variables with success probability $\rho_0$. Hence, $N_0$ is a negative binomial random variable with $C_0$ successes and success probability $\rho_0$.

Finally, using the lower bound on the mean of the minimum of a number of random variables of \citet{Aven1985} we get that
\begin{align} \label{boundOnN}
	\mathbb{E} N^* &= \mathbb{E} \min_{a \in \mathcal{A}_0} \{ N_a \} \ge \min_{a \in \mathcal{A}_0} \mathbb{E} N_a - \sqrt{ \frac A {A+1} \sum_{a \in \mathcal{A}_0} \text{Var} [N_a]} \nonumber\\
	&=N - \sqrt{ \frac A {A+1} } \sqrt { \sum_{a \in \mathcal{A}_0} N \frac {1 - \rho_a} {\rho_a}} = N - \sqrt{ N } K(\rho).
\end{align}
The result follows from combining \eqref{boundOnV} and \eqref{boundOnN}.\Halmos
\end{proof}

In terms of yield loss, our previous bound can be written as
$
J_N^* - J_N^B \le \sqrt{N} K(\rho) J_1^D,
$
achieving an $O(\sqrt N)$ loss w.r.t the optimal online policy. In particular, we may fix the capacity to impression ratio of each advertiser, and consider a sequence of problems in which capacity and impressions are scaled up proportionally according to $\rho$. Then, the yield under policy $\mu^B$ converges to the yield of the optimal online policy as $N$ goes to infinity.

A key observation in proving the last theorem was that the number of impressions in the left-over regime is $O(\sqrt N)$. In fact, using a Chernoff bound, we may show that the probability that the number of impressions in the left-over regime exceeds a fraction of the total impressions decays exponentially fast.

\begin{cor} \label{thm:left-over} The probability that the number of impressions in the left-over regime exceeds a fraction $\epsilon>0$ of the total impressions decays exponentially fast, as given by
    \begin{align*}
    \mathbb{P} \{ N - N^* \ge \epsilon N \} \le \sum_{a\in\mathcal{A}_0} \exp(-2 \epsilon^2 \rho_a N ).
    \end{align*}
\end{cor}
\begin{proof}{Proof.}
    We prove the complement, that is, the probability that $N^* \ge (1-\epsilon) N$ converges exponentially fast to one. Notice that $N^* \ge (1-\epsilon) N$ if and only if by time $(1-\epsilon)N$ the contract of each advertiser is not yet fulfilled ($S_{(1-\epsilon)N, a} < C_a$), and the point where all impressions need to be assigned to advertisers has not been reached ($S_{(1-\epsilon)N,0} < C_0$). Combining De Morgan's law and Boole's inequality we get that
    \begin{align*}
    \mathbb{P} \{ N^* \ge (1-\epsilon) N \}  = \mathbb{P} \{ S_{(1-\epsilon)N, a} < C_a \; \forall a \in \mathcal{A}_0 \}
    \ge 1 - \sum_{a\in\mathcal{A}_0} \mathbb{P} \{ S_{(1-\epsilon)N, a} \ge C_a\}.
    \end{align*}
    Recall that $S_{(1-\epsilon)N,0}$ is the sum of $(1-\epsilon)N$ independent Bernoulli random variables with success probability $\rho_a$. Hence, we conclude by applying Chernoff's bound to the each summand to obtain $\mathbb{P} \{ S_{(1-\epsilon)N, a} \ge C_a\} \le \exp(-2 \epsilon^2 \rho_a N )$.\Halmos
\end{proof}

The policy described in~\ref{sec:policy} is static in the sense that it does not react to changes in supply: the dual variables $v$ are computed at the beginning and remain fixed throughout the horizon. To address this issue, in practice, one would periodically resolve the deterministic approximation~\eqref{DUAL2}. Recently, \citet{JasinKumar2010} showed that carefully chosen periodic resolving schemes together with probabilistic allocation controls can achieve bounded yield loss w.r.t.~the optimal online policy. It is worth noting that those results do not directly apply to our setting: they consider a network RM problem with discrete choice, while our model deals with jointly distributed (and possibly continuous) placement qualities and AdX. Nevertheless, by periodically resolving the DAP one should be able to obtain similar performance guarantees for the yield loss of the control.

\section{Data Model and Estimation}\label{sec:data-model}


We have thus far assumed that any user could be potentially assigned to any advertiser. In practice, however, advertisers have specific targeting criteria. For instance, a guaranteed contract may demand for females with certain age range living in New York, while other contract may demand for males in California. In this section we give a parametric model based on our observation of real data, which takes into consideration that advertisers demand for particular \emph{user types} in their contracts. 

Instead of grouping user types according to their attributes, we aggregate user types that match the criteria of the same subset of advertisers. This has the advantage of reducing the space of types to a function of the number of advertisers (which is typically small in practice) rather then the number of possible types (which is potentially large). Hence, a user type is characterized by the subset of advertisers $T \subseteq \mathcal{A}$ that are interested in it. In the following, we let $\mathcal{T}$ be the support of the type distribution, and $\pi(T)$ the probability of an arriving impression being of type $T$. As before we assume that, across different impressions, types are independent and identically distributed. Given a particular type $T$, the predicted quality perceived by the advertisers within the type is modeled by the non-negative random vector $Q(T) = \{Q_{a}(T)\}_{a\in T}$.

Even if the total number of impressions suffices to satisfy the contracts, i.e. $\sum_{a \in \mathcal{A}} \rho_a \le 1$, the inventory may not be enough to satisfy the contracts targeting criteria. Our algorithm guarantees that the total number of impressions $C_a$ is always respected, yet some advertisers may be assigned impressions outside of their criteria. If an impression of type $T$ happens to be assigned to an advertiser $a\not\in T$, the publishers pays a nonnegative goodwill penalty $\tau_a$. These penalties allow the publisher to prioritize certain reservations, specially when the contracts are not feasible. Thus, the ex-ante distribution of quality is given by the mixture of the types distribution with mixing probabilities $\pi(T)$.
Notice that all our previous results hold if we apply the same analysis to the mixture distribution.

\subsection{Estimation}

Although the number of types may be exponential in $A$, in practice we observe that a linear number of them suffice to characterize 98\% of the inventory. We observe that the predicted quality perceived by the advertisers within a type is approximately log-normal. This can be seen in Figure~\ref{fig:corr-plot}, where the empirical distribution of log-quality is graphically represented for a type with two advertisers (data is log-transformed). The histograms on the diagonal show the marginal log-quality of each advertiser, which approximately resemble a normal curve. On the off-diagonals, scatter plots show the correlation between advertisers, which is strongly positive. In some sense this is expected, since many advertisers have similar targeting criteria.


\begin{figure}[t]
\centering
\includegraphics[width=6cm]{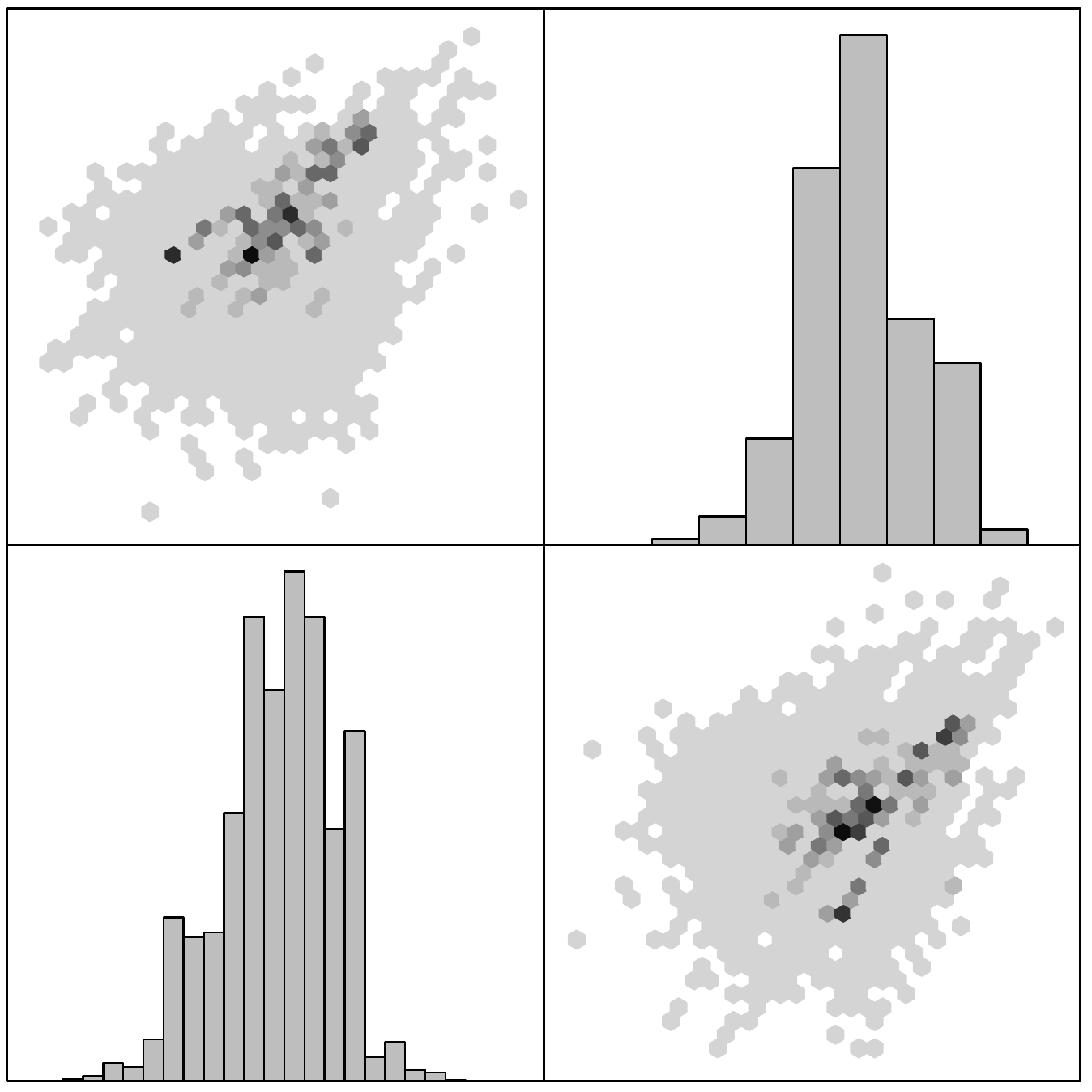}
\caption{Graphical representation of the empirical distribution of log-quality for a type with two advertisers (data is log-transformed).}\label{fig:corr-plot}
\end{figure}

Given a particular type $T$, we assume that quality follows a multivariate log-normal with mean vector $\mu_T$ and covariance matrix $\Sigma_T$ for the advertisers in the type, and takes a value of $-\tau_a$ for advertisers not in the type. The total distribution of quality is given by the mixture of these types distribution with mixing probabilities $\pi(T)$. Thus, we have that
\begin{align*}
	Q \sim \begin{cases}
	\ln \mathcal{N} (\mu_T, \Sigma_T), & \text{for } a \in T, \\
	-\tau_a, & \text{for } a \not\in T, \\
	\end{cases}
	\quad \text{w.p. } \pi(T).
\end{align*}

To perform the estimation we analyzed data from four different publishers for a consecutive period of seven days. First, the capacities of the reservations were used to compute the ratios $\rho$. Second, logs were analyzed to estimate the types' frequencies, and the parameters of the underlying log-normal distributions (using maximum likelihood estimation).

Bidding data from the same period of time was used to estimate the primitives of the AdX. With multiple bidders, AdX runs a sealed bid second-price auction. In this first approach to the problem, we assume that bids are independent of the quality of the impressions. We analyze the first and second highest bids for the inventory submitted to AdX. We denote by $\left\{(B_1^m, B_2^m)\right\}_{m=1,\ldots,M}$ the sampled highest and second highest bids from the exchange. Sample data is used to compute the two primitives of our model: (i) the complement of the quantile of the highest bid $p(s)$, and (ii) the revenue function of $r(s)$. Both functions are estimated on a uniform grid $\{s_j\}_1^{100}$ of survival probabilities in the $[0,1]$ range.

First, for each point in the grid $j$, the price $p_j = p(s_j)$ is estimated as the $(1-s_j)$-th population quantile of the highest bid. Then, using  sampled bids, we estimate the revenue function w.r.t. to prices at the grid points as
\begin{align}\label{r-estimated}
	r(p_j) = \frac 1 M \sum_{m=1}^M \mathbf{1} \{B_1^m \ge p_j\} \max\{B_2^m, p_j\}
\end{align}
Finally, the revenue function is obtained by composing \eqref{r-estimated} and $p(s)$. AdX data is available only for the first two publishers.

Figure \ref{fig:est-type-graph} describes Instance 1, a publisher with 4 types and 3 advertisers. The estimated survival probability and revenue function for the publisher is shown in Figure \ref{fig:est-adx}. The parameters for the remaining publishers are available at the webpage of the first author.

\begin{figure}[t]
\centering
    \subfloat[User type-advertiser graph.]{\includegraphics[width=0.45\textwidth]{figure-5.mps}\label{fig:est-type-graph}} \hspace{0.08\textwidth}
    \subfloat[Estimated survival probability and revenue function for AdX.]{\includegraphics[width=0.45\textwidth]{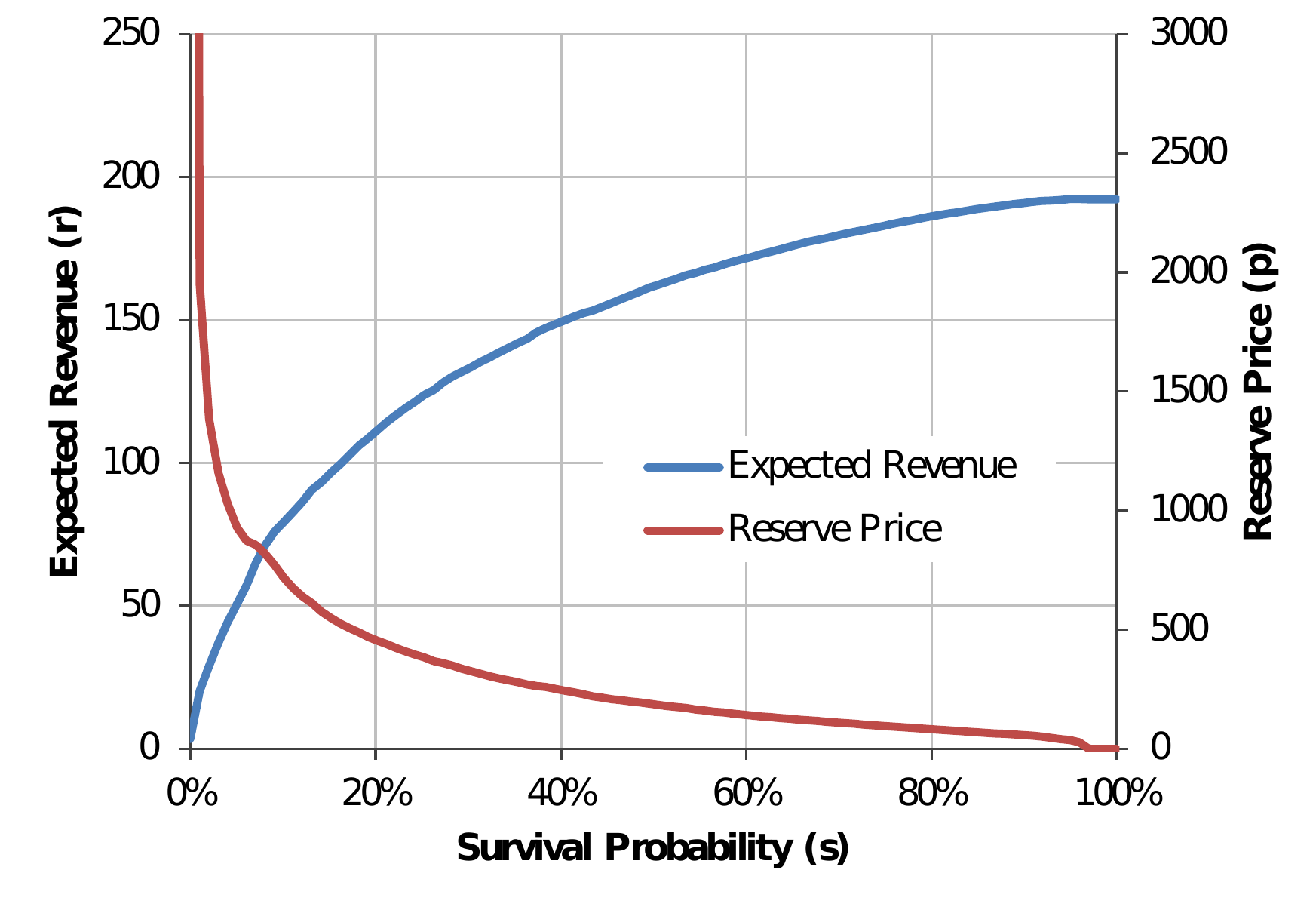} \label{fig:est-adx}}
    \hspace{0.08\textwidth}
    \\
    \subfloat[Parameters of the distribution of log-quality.]{
    \begin{tabular}[b]{c|cccc}
    Type & Ads & $\pi(T)$ & $\mu_T$ & $\Sigma_T$ \\ \hline \\[-2ex]
    $T_1$ & $\{1,2,3\}$ & 0.2 &
    $\left(\begin{smallmatrix} 7.8155 \\ 7.8155 \\ 7.8155 \end{smallmatrix}\right)$ &
    $\left(\begin{smallmatrix}0.3 & 0.1 & 0.1 \\ 0.1 & 0.3 & 0.1 \\ 0.1 & 0.3 & 0.1 \end{smallmatrix} \right)$ \\[1.5ex]
    $T_2$ & $\{1,2\}$ & 0.3 &
    $\left(\begin{smallmatrix} 6.6755 \\ 7.0655 \end{smallmatrix}\right)$ &
    $\left(\begin{smallmatrix} 0.3180 & 0.1649 \\ 0.1649 & 0.3602 \end{smallmatrix} \right)$ \\[1ex]
    $T_3$ & $\{2,3\}$ & 0.1 &
    $\left(\begin{smallmatrix} 6.6355 \\ 7.8055 \end{smallmatrix}\right)$ &
    $\left(\begin{smallmatrix} 0.4347 & 0.2357 \\ 0.2357 & 0.4367 \end{smallmatrix} \right)$ \\[1ex]
    $T_4$ & $\{1,3\}$ & 0.4 &
    $\left(\begin{smallmatrix} 7.2155 \\ 6.9155 \end{smallmatrix}\right)$ &
    $\left(\begin{smallmatrix} 0.23 & 0.05 \\ 0.05 & 0.40  \end{smallmatrix} \right)$ \\[1ex]
    \end{tabular}
    }
    \caption{Description of Instance 1.}
\end{figure}


\section{Experimental Results}\label{sec:experiments}

Two experiments were conducted to study our algorithm. First we study the impact of introducing an AdX on the publisher's yield.
Second, we compare the previously known primal-dual approach to ad allocation that is non-parametric to our approach here which is
parametric.

\subsection{Impact of AdX}\label{sec:impact-adx}

This first experiment explores the potential benefits of introducing an AdX, and how the publisher can take advantage of it. We study the impact of the trade-off parameter $\gamma$ on both objectives, that is, the quality of the impressions assigned to the advertisers, and the revenue from AdX. The limiting choices of $\gamma = 0$, and $\gamma = \infty$ are of particular interest. The first choice represents the case where the publisher disregards the quality of the impressions assigned to the advertisers, and strives to maximize the revenue extracted from AdX. Here the publisher strategically picks the reserve price so that just enough impressions are rejected to satisfy the contracts. In the second choice, the publishers prioritizes the quality of the impressions assigned, and submits the remanent inventory to AdX. We use this case as the baseline to which we compare our method.

The experiment was conducted as follows. First, we set up a grid on the trade-off parameter $\gamma$. Then, we solve the publisher's problem as given in \eqref{DUAL2}. The resulting policies are evaluated using a fluid limit (see \ref{sec:fluid-limit}). Table \ref{table:adx-results} reports the expected quality and revenue for different choices of $\gamma$. Figure \ref{fig:chartAdXratio2} plots the quality, and revenue relative to the baseline case; as a function of $\gamma$. In Figure \ref{fig:chartAdXpareto2} we plot, in a quality vs.\ revenue graph, the objective values of the optimal solutions for the different choices of $\gamma$, together with the Pareto frontier.

\begin{figure}[t]
    \centering
    \subfloat[Ratios to baseline]{\includegraphics[height=5.5cm]{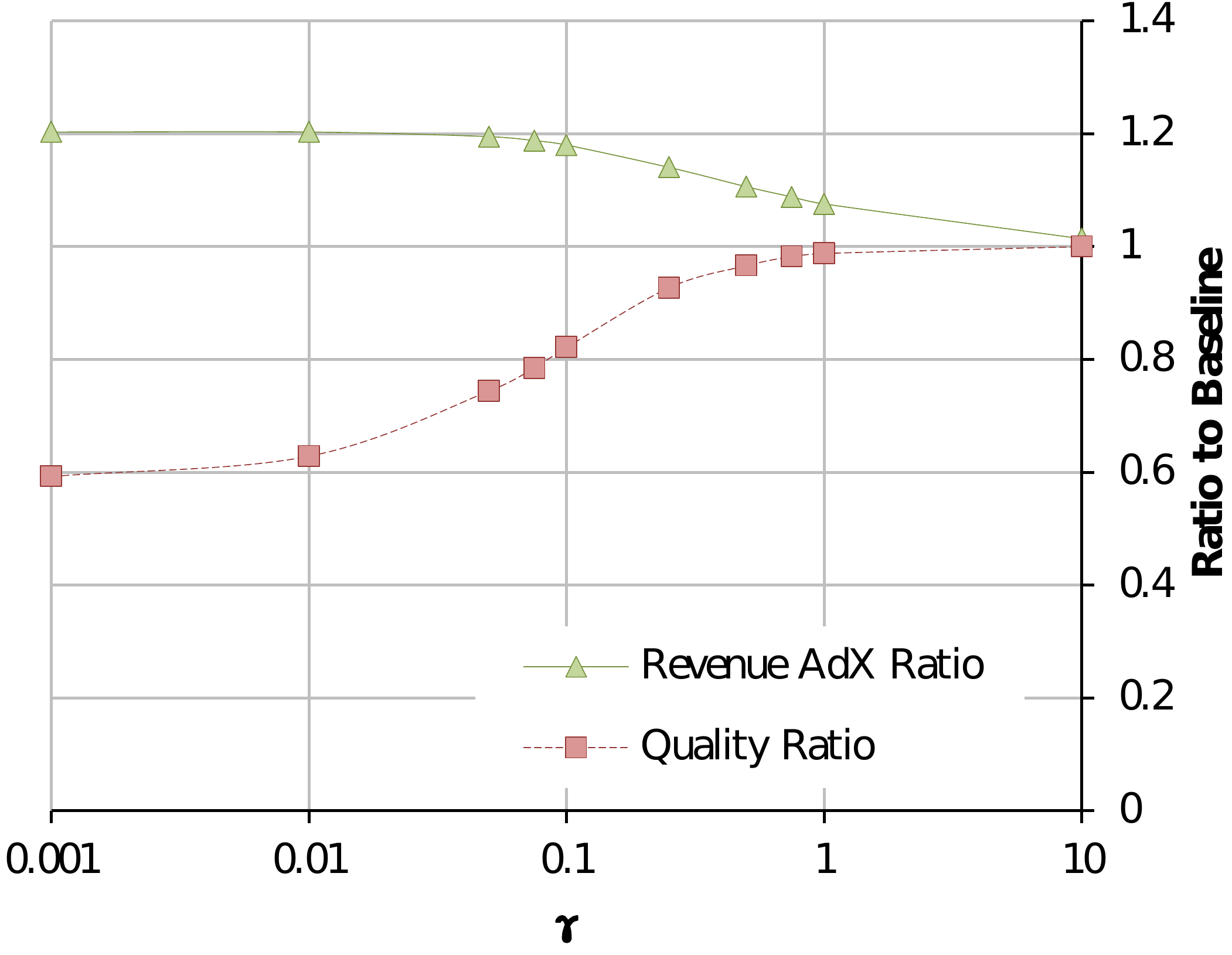}\label{fig:chartAdXratio2}}
    \hspace{0.08\textwidth}
    \subfloat[Pareto Frontier]{\includegraphics[height=5.5cm]{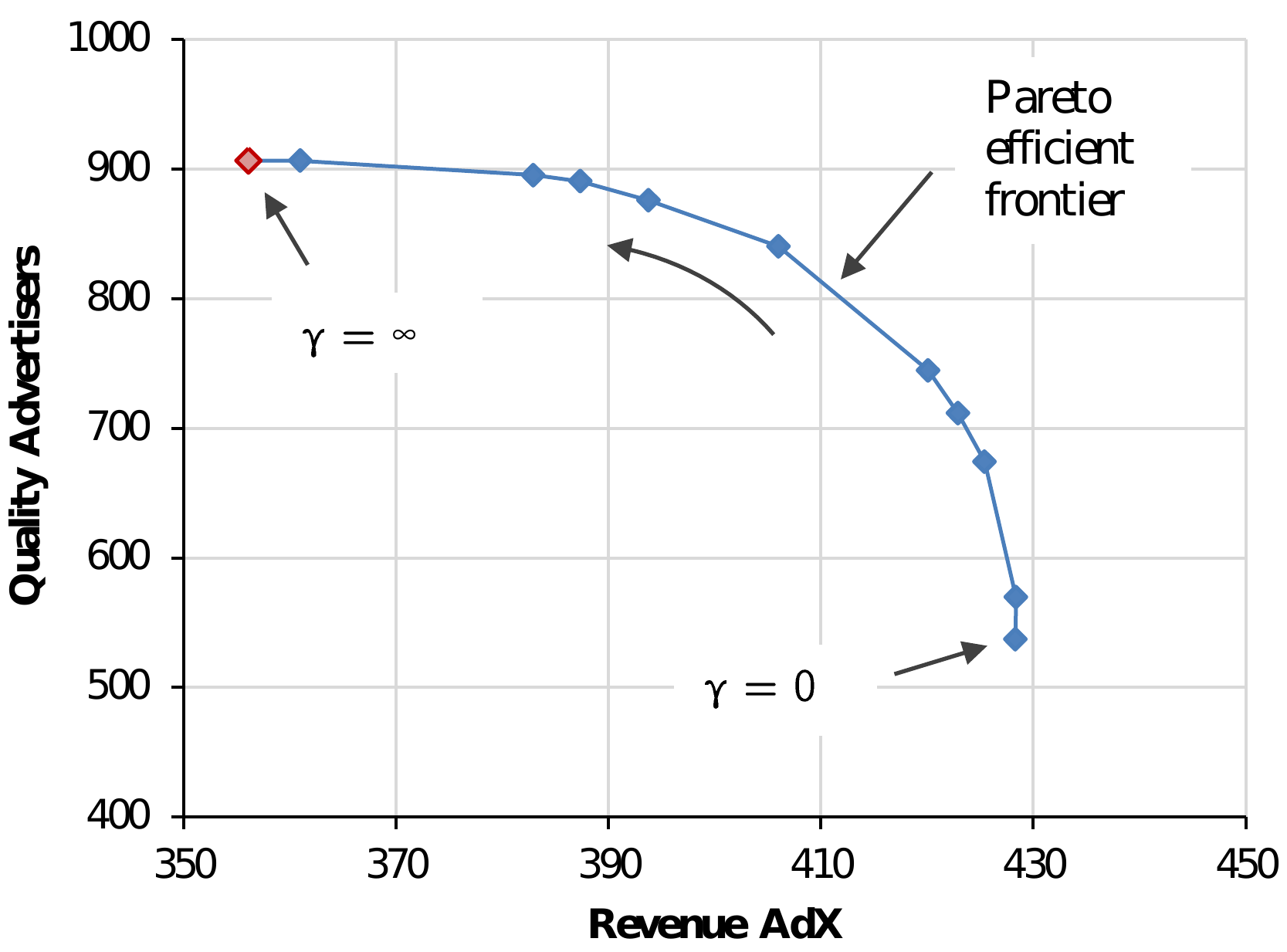}\label{fig:chartAdXpareto2}}
    \footnotesize
    \caption{Figure \ref{fig:chartAdXratio2} plots the quality and revenue relative to the case $\gamma=\infty$, as a function of $\gamma$. Figure \ref{fig:chartAdXpareto2} plots, in a quality vs.\ revenue graph, the objective values of the optimal solutions for the different choices of $\gamma$, together with the Pareto frontier. Both plots are for Instance 2, given in~\S\ref{sec:experiments}.}
    \label{fig:chartAdX}
\end{figure}

\begin{table}
\centering
\fontfamily{phv}\fontseries{mc}\selectfont
\renewcommand{\tabcolsep}{1ex}
\textbf{Instance 1}\\
\vspace{1ex}
\begin{tabular}{c|cccccccccccc}
$\gamma$	& 0	& 0.001	& 0.01	& 0.05	& 0.075	& 0.1\\ \hline
Yield	& 110.94	& 112.78	& 128.81	& 202.46	& 249.11	& 296.95	\\
Quality	& 1107.32	& 1779.07	& 1801.00	& 1864.95	& 1891.19	& 1913.78	\\
Revenue	& 110.94	& 111.00	& 110.80	& 109.21	& 107.27	& 105.57	\\[3ex]
$\gamma$ & 0.25	& 0.5	& 0.75	& 1	& 10	& $\infty$ \\ \hline
Yield   & 590.54	& 1098.43	& 1608.80	& 2122.99	& 20764.82 	& $\infty$\\
Quality	& 1998.31	& 2044.89	& 2055.72	& 2061.33	& 2072.22	& 2075.52\\
Revenue	& 90.97	    & 75.98 	& 67.00 	& 61.66 	& 42.61 	& 38.48\\
\end{tabular}
\vspace{4ex}

\textbf{Instance 2}\\
\vspace{1ex}
\begin{tabular}{c|cccccccccccc}
$\gamma$	& 0	& 0.001	& 0.01	& 0.05	& 0.075	& 0.1\\ \hline
Yield	& 428.73	& 429.02	& 434.22	& 459.57	& 477.39	& 495.66	\\
Quality	& 483.02	& 545.27	& 573.23	& 676.34	& 720.31	& 752.35	\\
Revenue	& 428.73	& 428.47	& 428.49	& 425.75	& 423.37	& 420.42	\\[3ex]
$\gamma$ & 0.25	& 0.5	& 0.75	& 1	& 10	& $\infty$ \\ \hline
Yield   & 617.04	& 834.39	& 1056.05	& 1279.88	& 9425.63	& $\infty$\\
Quality	& 843.47	& 880.69	& 891.49	& 896.89	& 906.46	& 907.05\\
Revenue	& 406.17	& 394.05	& 387.43	& 382.99	& 360.99	& 356.11\\
\end{tabular}
\caption{Expected yield, advertisers' quality and revenue from AdX for two instances, and different choices of $\gamma$.}
\label{table:adx-results}
\end{table}

\paragraphtitle{Discussion.} Results confirm that, as we increase the trade-off parameter $\gamma$, the quality of the impressions assigned to the advertisers increases, while the revenue from AdX subsides. Interestingly, starting from the baseline case that disregards AdX ($\gamma = \infty$), we observe that the revenue from AdX can be substantially increased by sacrificing a small fraction of the overall quality of the impressions assigned. For instance, by exploiting strategically the AdX, the publisher can increase AdX's revenue by 8\% by giving up only 1\% quality. Conversely, starting from the case that disregards the advertiser's quality ($\gamma = 0$), the publisher can raise the quality in a large amount at the expense of a small decrease in AdX's revenue.

Alternatively, the previous analysis can be understood in terms of the Pareto frontier. Results show that the Pareto frontier is highly concave, relatively horizontally flat around $\gamma = \infty$, and vertically flat around $\gamma = 0$. This explains the huge marginal improvements at the extremes. There are several advantages to the quality vs.\ revenue representation. First, the Pareto frontier allows for quick grasp of the nature of the operation. When the publisher's current operation is sub-optimal, its performance point should lie in the interior of the frontier. In this case, the Pareto frontier allows the publisher to measure its efficiency, and quantify the potential benefits an optimal policy may introduce. Second, when the choice of the trade-off parameter is not clear, the publisher may impose a lower bound on the overall quality of the impressions, and instead maximize the total revenue from AdX. The efficient frontier provides the maximum attainable revenue, and the proper $\gamma$ to achieve the quality constraint.

\subsection{Comparison with the Primal-Dual Approach}\label{sec:PD-compare}

In this second experiment we study the performance the our algorithm, and contrast it with a Primal-Dual (PD) method.
This experiment is discussed in detail in~\S\ref{sec:withoutAdX}, but we give the full experiments and discussion here for the benefit of the reader. Since no existing PD method is known yet for the AdX problem, we consider instead the case with no AdX. The Primal-Dual approach~\citep{DevenurHayes2009}, uses a sample from data to estimate the dual variables and uses it in a bid-price control policy. In contrast, our algorithm, as stated, assumes the parameters of the quality distribution are known, and uses that to estimate the dual variables. So we do not need to use a sample.  Of course in practice, the parameters need to be learned, and so we would need to use a sample of the data in order to learn them; but in many settings (including online advertising) it is reasonable to assume that we at least know the {\em form} of the distribution (e.g., normal, exponential, Zipf), albeit not the specific parameters (mean, variance, covariance, etc.).  The techniques in~\cite{DevenurHayes2009} are powerful {\em because} they don't need to assume anything about the distribution, but it is important to ask what can be gained from knowing the {\em form} of the distribution, which is what we do in the remainder of this section.

In order to objectively assess the performance of our algorithm we adopt the user type model described in \S\ref{sec:data-model} as a generative model. The generative model is used to generate sample data on which both our algorithm and a PD method are tested. The advantages of adopting a generative model are twofold. First, it allows us to compute the truly optimal policy $\mu^\OPT$. Second, the true performance of any policy can be evaluated efficiently using a fluid limit (see Section~\ref{sec:fluid-limit}).

The computational experiment is conducted as follows. First, a \emph{training} data set of $M$ impressions is generated. We denote the sampled quality vectors by $\{q_m\}_{m=1}^M$. Then, we estimate the parameters of the model on the training set as follows. For each type we estimate the type probabilities $\hat{\pi_T}$; and mean $\hat{\mu_T}$, and covariance matrix $\hat{\Sigma_T}$ of the logarithm of the qualities. Next, the dual problem \eqref{DUAL2} is solved on the estimated parametric model using a Gradient Descent Method as described in \S\ref{sec:computation}. Note that, since no AdX is considered, the maximum expected revenue function $R(\cdot)$ is the identity. Using the optimal solution $v^\EST$ we construct a policy, which be refer as $\mu^\EST$.

Simultaneously, we employ the PD method on the training data. The PD method amounts to solving a sample average approximation of problem \eqref{DUAL2}, which results in the following linear program
\begin{align}
    \min_{v,\lambda} & \;  \frac 1 M \sum_{m=1}^M \lambda_m  + \sum_{a \in \mathcal{A}} \rho_a v_a \label{eq:DevenurHayes2009}\\
    \text{s.t. } & \lambda_m + v_a \ge q_{m,a},	\qquad \forall m,a \nonumber \\
	& \lambda_m \ge 0 \qquad \forall m. \nonumber
\end{align}
The linear program is solved using CPLEX 12. Again, using the dual optimal solution $v^\PDual$ we construct a policy $\mu^\PDual$.

Afterwards, we assess the performance of both policies using a fluid limit. These steps are replicated on $50$ different training sets. Table \ref{table:PD-comp} reports the average results over the training sets for different sizes of training sets, and instances. Plots of the results for a given instance are shown in Figure \ref{fig:chart63}.

\begin{table}
\centering
\fontfamily{phv}\fontseries{mc}\selectfont
\renewcommand{\tabcolsep}{2ex}

\textbf{Instance 1} (A = 3, T = 4, OPT = 2075.09)\\
\vspace{1ex}
\begin{tabular}{c|cc|cc}
\multirow{2}{8ex}{Training Set Size}	& \multicolumn{2}{c|}{EST}	& 	\multicolumn{2}{c}{PD}\\
	& mean	& std.dev.	& mean.	& std.dev \\ \hline
100	& 2004.16 (3.42\%)	& 33.978	& 1990.32 (4.08\%)	& 37.552 \\
1000	& 2053.41 (1.04\%)	& 10.008	& 2047.92 (1.31\%)	& 12.365 \\
2500	& 2065.12 (0.48\%)	& 4.956	& 2062.76 (0.59\%)	& 5.838 \\
5000	& 2068.44 (0.32\%)	& 3.681	& 2066.99 (0.39\%)	& 4.224 \\
\end{tabular}
\vspace{4ex}

\textbf{Instance 2} (A = 6, T = 10, OPT = 907.44)\\
\vspace{1ex}
\begin{tabular}{c|cc|cc}
\multirow{2}{8ex}{Training Set Size}	& \multicolumn{2}{c|}{EST}	& 	\multicolumn{2}{c}{PD}\\
	& mean	& std.dev.	& mean.	& std.dev \\ \hline
1000	& 889.58 (1.97\%)	& 8.861	& 882.77 (2.72\%)	& 12.829\\
2500	& 894.43 (1.43\%)	& 7.485	& 887.64 (2.18\%)	& 10.418\\
5000	& 898.59 (0.98\%)	& 5.231	& 892.51 (1.65\%)	& 7.625\\
10000	& 901.13 (0.70\%)	& 3.588	& 897.42 (1.10\%)	& 4.692\\
25000	& 904.69 (0.30\%)	& 1.712	& 901.97 (0.60\%)	& 2.720\\
50000	& 905.03 (0.27\%)	& 1.267	& 903.44 (0.44\%)	& 1.567\\
\end{tabular}
\vspace{4ex}

\textbf{Instance 3} (A = 17, T = 15, OPT = 894.82)\\
\vspace{1ex}
\begin{tabular}{c|cc|cc}
\multirow{2}{8ex}{Training Set Size}	& \multicolumn{2}{c|}{EST}	& 	\multicolumn{2}{c}{PD}\\
	& mean	& std.dev.	& mean.	& std.dev \\ \hline
2500	& 859.83 (3.91\%)	& 9.937	& 849.44 (5.07\%)	& 14.615\\
5000	& 868.61 (2.93\%)	& 5.870	& 861.06 (3.77\%)	& 7.954\\
10000	& 877.59 (1.92\%)	& 5.226	& 873.46 (2.39\%)	& 6.577\\
25000	& 884.04 (1.20\%)	& 2.585	& 881.13 (1.53\%)	& 3.747\\
50000	& 887.34 (0.84\%)	& 1.926	& 885.11 (1.08\%)	& 2.728\\
\end{tabular}
\vspace{4ex}

\textbf{Instance 4} (A = 14, T = 10, OPT = 928.76)\\
\vspace{1ex}
\begin{tabular}{c|cc|cc}
\multirow{2}{8ex}{Training Set Size}	& \multicolumn{2}{c|}{EST}	& 	\multicolumn{2}{c}{PD}\\
	& mean	& std.dev.	& mean.	& std.dev \\ \hline
2500	& 892.55 (3.90\%)	& 12.886	& 888.88 (4.29\%)	& 13.427\\
5000	& 903.04 (2.77\%)	& 8.537	& 901.79 (2.90\%)	& 10.277\\
10000	& 911.25 (1.88\%)	& 6.951	& 909.96 (2.02\%)	& 6.935\\
25000	& 917.30 (1.23\%)	& 3.353	& 915.81 (1.39\%)	& 3.705\\
50000	& 921.36 (0.80\%)	& 2.668	& 920.11 (0.93\%)	& 2.716\\
\end{tabular}
\caption{Experimental results comparing the performance of our parametric method (EST) with the non-parametric primal-dual method (PD). No AdX present in this experiment.}
\label{table:PD-comp}
\end{table}

\begin{figure}[t]
    \centering
    \subfloat[Average Yield]{\includegraphics[width=0.45\textwidth]{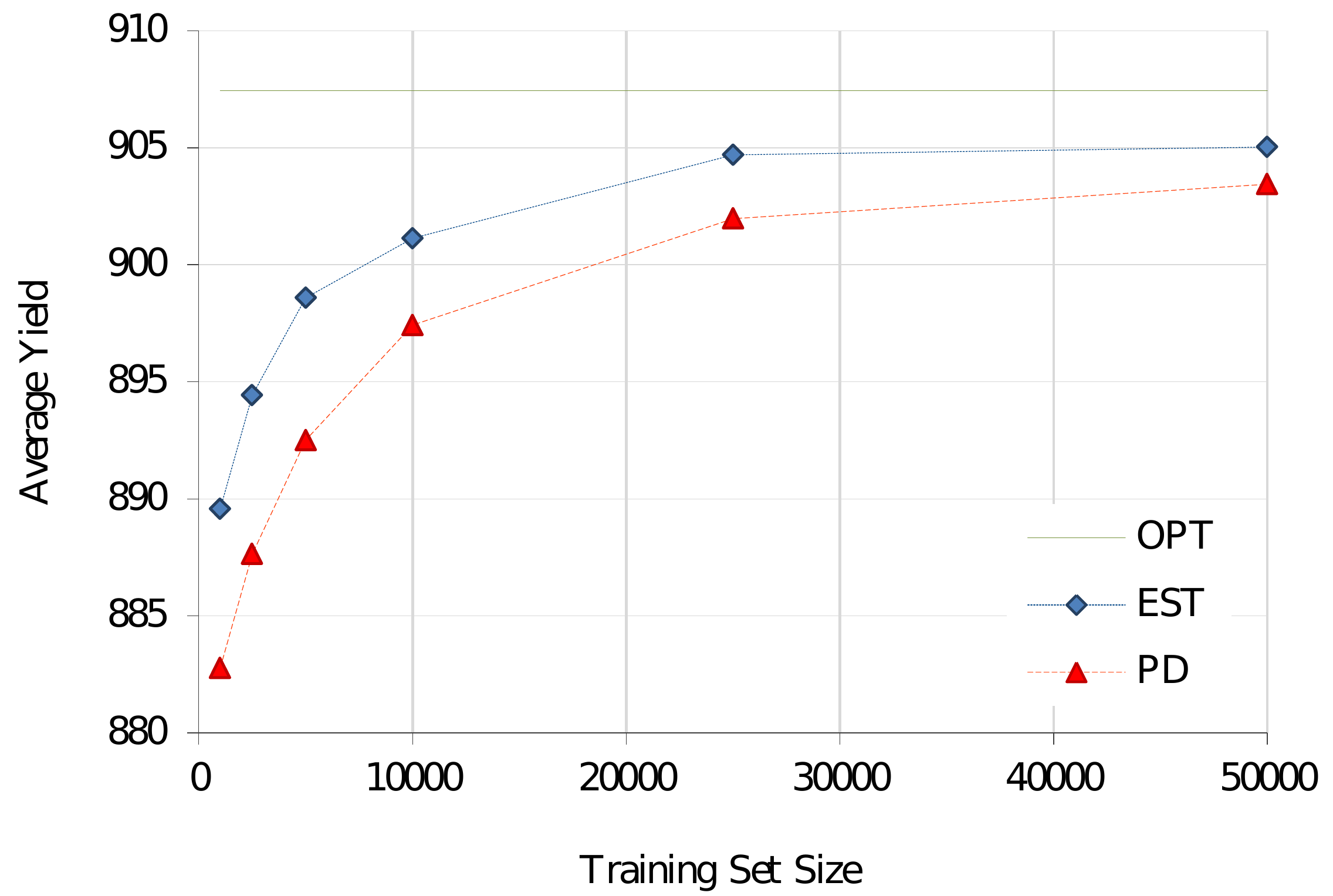}}
    \hspace{0.08\textwidth}
    \subfloat[Std. Dev. of Yield]{\includegraphics[width=0.45\textwidth]{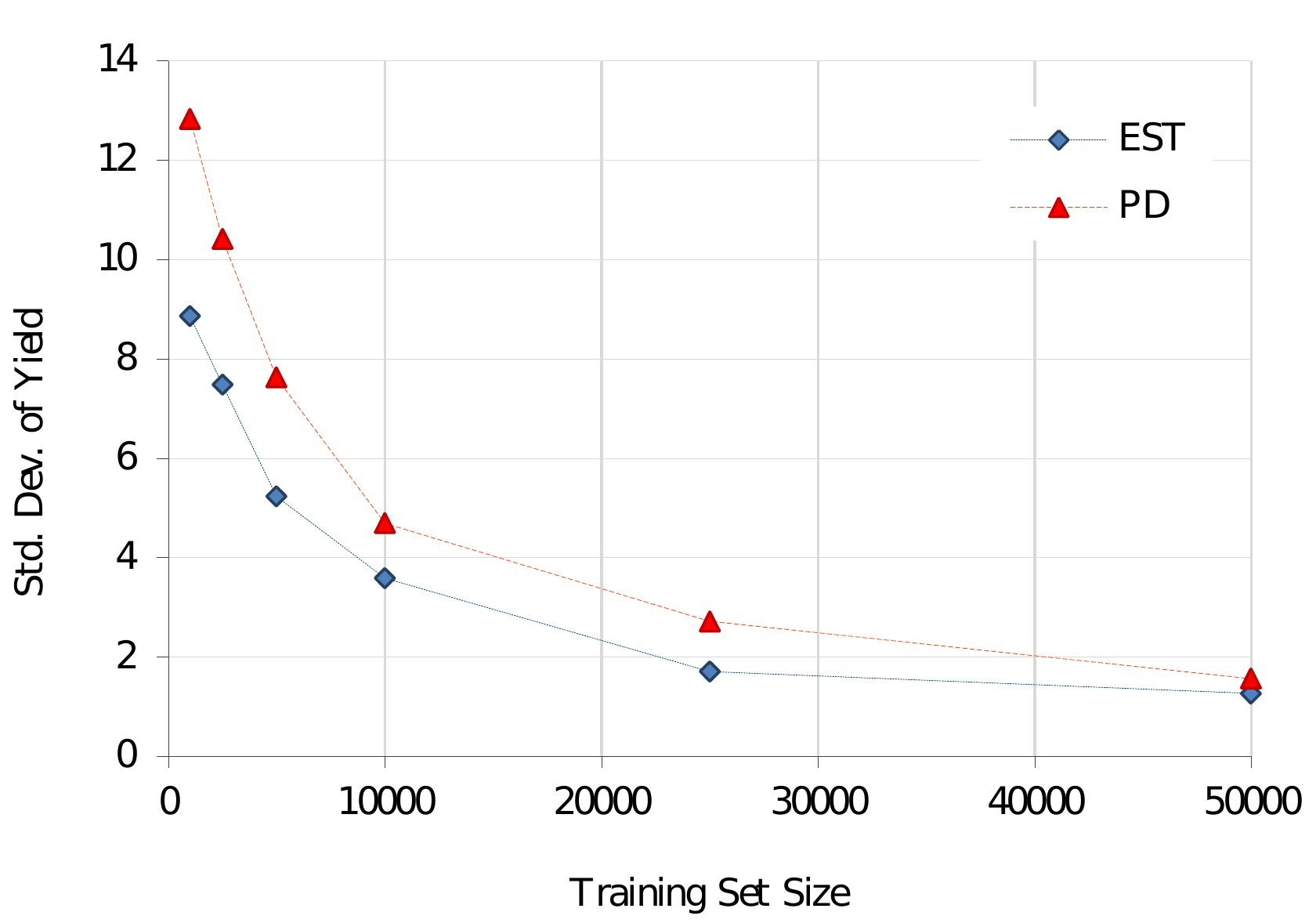}}
    \caption{Average (a) and standard deviation (b) of yield as a function of training set sample size;  results are shown for the parametric method (EST) based on our policy $\mu$ and the primal-dual method (PD) as in~\cite{DevenurHayes2009}. Both policies converge to the optimal yield, but EST converges faster, and with less variance.}
    \label{fig:chart63}
\end{figure}

\paragraphtitle{Discussion.} Results show that for both algorithms, as the size of the training set increases, the optimality gap decreases at a rate of $O(M^{\frac 1 2})$. However, the parametric method performs uniformly better that the non-parametric PD method. Additionally, the variability across different training sets diminishes as the size of the training set increases. Indeed, we observe that the standard deviation over training sets converges to zero for both methods, but the convergence is faster for the parametric one. In some sense this is expected, since the true data model follows exactly the distributional assumptions. However, the PD method is expected to be more robust to model misspecification.

Another experiment, though results are not reported, was conducted to test the strength of the parametric method on real data. We observed that, when the training set is small (around thousands), the parametric method performs better than the non-parametric one. However, as the sample size increases the non-parametric method outperforms the other. The rationale for this behavior is that, when data is scarce, the parametric method can exploit the distributional assumptions to reconstruct a fair representation of the data. However when the training set is larger, the fit of our model to real data is not perfect, and the non-parametric method can withstand deviations more robustly.

\section{Extensions}
In this section we consider a number of extensions of the model and policy from the previous section.

\subsection{Target Quality Constraints}\label{sec:quality-constraints}

In section~\ref{sec:obective} we discussed an alternate formulation in which the publisher imposes a minimum overall quality for the impressions assigned to the advertisers. This might be more natural for some publishers; they might feel more comfortable specifying target quality constraint than picking a Lagrange multiplier to weight the impact of quality in the objective. Additionally, in some settings the advertisers themselves might demand that certain level of quality is guaranteed.

In the following, we impose that the average quality of the impressions assigned to advertiser $a$ is larger or equal than a threshold value $\ell_a$. Now, the publisher would strive to maximize the revenue from AdX, while complying with the target quality constraints, and the contractual obligations. The one-impression DAP would be similar, except that the objective only accounts for AdX's revenue, and the inclusion of the constraints
\begin{align}
    \mathbb{E} \left[ i_{a}(Q) Q_{a} \right] \ge \ell_a, \qquad \forall a \in \mathcal{A}. \label{DAPQ:quality}
\end{align}

We attack the problem, as done before, by considering its dual. Let $\gamma_a \ge 0$ be the Lagrange multiplier associated to \eqref{DAPQ:quality}. As a side note, problem \eqref{DAP} can be interpreted as the Lagrange relaxation of our new problem w.r.t.~the target quality constraints, and the dual variables $\gamma$ as the shadow price of the target quality constraints. The new constraints preserve the convexity of the program, and strong duality still holds. Following the same steps, we obtain the new dual problem
\begin{align*} 
	\min_{\gamma\ge0,v} & \; \left\{ \mathbb{E} R\left(\max_{a \in \mathcal{A}_0} \{\gamma_a Q_a - v_a\}\right) + \sum_{a \in \mathcal{A}} v_a \rho_a - \gamma_a \ell_a \right\},
\end{align*}
which still is a convex minimization problem. The publisher might now jointly optimize over $v$, and $\gamma$ to construct a provably good policy. Additionally, in a similar fashion to Proposition~\ref{prop:gradient}, we may compute the directional derivative of the objective w.r.t.~the dual variables $\gamma$.

Regarding the performance the bid-price control $\mu^B$, Theorem~\ref{thm:main} still holds, and the policy asymptotically attains the optimal revenue from AdX, while complying with the delivery targets. However, we still need to argue about the expected average quality assigned to the advertisers. Unfortunately, for those advertisers whose constraint \eqref{DAPQ:quality} is binding, our algorithm might not attain the desired quality target. 
Nevertheless, from our asymptotic analysis we may show that the expected average quality is lower bounded by
\begin{align*}
	\mathbb{E} \left[ \frac 1 N \sum_{n=1}^N \left[1-s_n^\mu(Q_n)\right] I_{n,a}^\mu(Q_n) Q_{n,a} \right]
    \ge \frac {\mathbb{E} N^* } N \mathbb{E} \left[ i^*_{a}(Q) Q_{a} \right] \ge \left(1 - \frac 1 {\sqrt N} K(\rho)\right) \ell_a.
\end{align*}
Hence, for advertisers with binding constraint \eqref{DAPQ:quality}, albeit not feasible, the expected average quality becomes arbitrary close to the threshold value as the number of impressions in the horizon increases. On the other hand, for the remaining advertisers whose target quality constraint is not binding, the expected average quality will surpass the threshold for suitably large $N$.


\subsection{AdX with Multiple Bidders}\label{sec:AdX-multiple-bidders}

Here we generalize our results to the case where multiple buyers participate in the Ad Exchange. We model AdX as an auction with $K$ risk neutral buyers. The publisher believes that individual valuations are drawn independently from the same distribution with c.d.f $F(\cdot)$, density $f(\cdot)$, and support $[p_0,p_\infty]$. Moreover, we assume that the distribution of the values have increasing failure rates, are absolutely continuous and strictly monotonic. The publisher must choose the reservation price $p$ that maximizes her expected revenue given that her value for the impression is $c\ge0$. As before, we denote by $R(c)$ the optimal expected revenue of the publisher.

\cite{Myerson1981} argued that under our assumptions the optimal mechanism is a Vickrey or second-price sealed-bid auction. Moreover, it is known that in such auctions bidding the true valuation is a dominant strategy for the buyers, and that the optimal reservation price $p^*(c)$ is independent of the number of buyers~\citep{LaffontMaskin1980}.

Let $B_{1:K}$ and $B_{2:K}$ be the order statistics which denote the highest and the second highest bid respectively. Given a reserve price $p$, the item is sold if $B_{1:K} \ge p$, i.e., there is some bid higher than the reserve price. The winning buyer pays the second highest bid, or alternatively $\max\{B_{2:K},p\}$, since the seller should receive at least the reserve price $p$. Therefore, the publisher's maximization problem is
\begin{align*}
	R(c) &= \max_{p \ge 0} \mathbb{E} \left[ \mathbf{1}\{B_{1:K} \ge p\} \max\{B_{2:K},p\} + \mathbf{1}\{B_{1:K} < p\} c \right].
\end{align*}
Notice that the setup of Section \ref{sec:AdX} can be consider as a particular case of a second-price auction in which we have only one bidder and $B_{2:K} = 0$.

Recall that, instead of reserve prices, we casted our problem in terms of survival or winning probabilities. Then, letting $s$ be the probability than the impression is sold, we have that $s = \mathbb{P} \{ B_{1:K} \ge p \} = 1 - F^K(p)$ since valuations are i.i.d. Conversely, the reserve price as a function of the survival probability is given by $p(s) = \bar{F}^{-1} (1 - (1-s)^{1/K})$, which is well-defined due to the strict monotonicity of the c.d.f. In terms of survival probabilities, the problem is now
\begin{align*}
	R(c) &= \max_{0 \le s \le 1} r(s) + (1-s) c,
\end{align*}
where we defined the revenue function as $r(s) = r(p(s))$, and $r(p) = \mathbb{E} \left[ \mathbf{1}\{B_{1:K} \ge p\} \max\{B_{2:K},p\} \right]$.

The next proposition shows that the revenue function is regular, and as a consequence all previous results hold for the case with multiple bidders.

\begin{prop} Under the previous assumption, the revenue function $r(s)$ is regular. Moreover, the optimal reserve price $p^*(c)$ solves
	\begin{align*}
		\frac {\bar{F}(p)} {f(p)} = p - c,
	\end{align*}
when $c \in [p_0 - 1/f(p_0), p_\infty]$. When the opportunity cost is higher than the null price $(c > p_\infty)$, the publisher bypasses the exchange $(p^*(c) = p_\infty)$. Finally, when the opportunity cost is low enough $(c<p_0 - 1/f(p_0))$, the impression is kept by the highest bidder $(p^*(c) = p_0)$.
\end{prop}
\begin{proof}{Proof.}
The joint distribution of $B_{1:K}$ and $B_{2:K}$ has a density function~\citep{LaffontMaskin1980}
$$
	f(b_1,b_2) = \begin{cases} K(K-1) F(b_2)^{K-2}f(b_1)f(b_2) & \text{if } b_1 \ge b_2 \\ 0 & \text{otherwise} \end{cases}.
$$
Then, we have that
\begin{align*}
	r(p) &= \mathbb{E} \left[ \mathbf{1}\{B_{2:K} \ge p\} B_{2:K} + p  \mathbf{1}\{B_{1:K} \ge p, B_{2:K} < p\} \right] \\
	&= \int_p^\infty \int_p^{b_1} b_2 f(b_1, b_2) \text{ d}b_2 \text{ d}b_1 + p \int_p^\infty \int_0^p f(b_1, b_2)  \text{ d}b_2 \text{ d}b_1 \\
	&= K(K-1) \int_p^\infty b_2 F(b_2)^{K-2} f(b_2) (1 - F(b_2)) \text{ d}b_2 + K p F(p)^{K-1} (1-F(p))
\end{align*}
Continuity of $r(s)$ follows because the p.d.f. is continuous, and $p(s)$ is continuous (if F not strictly monotone, the inverse may have jumps). Additionally, we may bound the revenue by
\begin{align*}
	r(p) &\le \mathbb{E} \left[ \mathbf{1}\{B_{1:K} \ge p\} B_{1:K} \right] \le K \mathbb{E} \left[ \mathbf{1}\{B \ge p\} B \right] \le K \mathbb{E} B < \infty,
\end{align*}
the first inequality follows because $B_{1:K}$ is the maximum, the second because any order statistic is upper bounded by the sum of the bids, and the fourth because bids are integrable. Moreover, integrability of $B$ implies that $\lim_{p\rightarrow\infty} r(p) = 0$.

Next, we turn to the concavity of $r(s)$. Differentiating w.r.t to $p$ we get
\begin{align*}
	\frac {d r} {d p} = K F(p)^{K-1} ( \bar{F}(p) - p f(p) ).
\end{align*}
Then, using the fact that $\frac {d s} {d p} = - K F(p)^{1-k} / f(p)$ we get from the composition rule that
\begin{align*}
	\frac {d r} {d s} = \left. \frac {d r} {d p} \right|_{p(s)} \frac {d p} {d s} =  p(s) - \frac 1 {h(p(s))},
\end{align*}
where $h(p) = f(p) / \bar{F}(p)$ is the hazard rate of the bidder's valuation. Because $p(s)$ is non-increasing in $s$ and the $h(p)$ is non-decreasing in $p$, we conclude that $\frac {d r} {d s}$ is non-increasing. Thus, the revenue function is concave.


Finally, notice that the that derivative of the objective w.r.t to $s$ is
\begin{align}\label{R-deriv-obj}
    p(s) - \frac 1 {h(p(s))} - c,
\end{align}
which is non-increasing. When $c > p_\infty$ we have that \eqref{R-deriv-obj} is negative, so $s^*(c) = 0$ and $p^*(c) = p_\infty$. Similarly, when $c<p_0 - 1/h(p_0)$ we that \eqref{R-deriv-obj} is positive, so $s^*(c) = 1$ and $p^*(c) = p_0$.\Halmos
\end{proof}



\subsection{AdX with User Information}\label{sec:AdX-user-info}

In most systems, the publisher shares some user information with the exchange. In turn, the exchange may partially disclose the user information to their advertisers. The advertisers may react to this information, and bid strategically~\citep{Muthu2009}. In this section we extend our model to the case when the bids from AdX are correlated with the quality of the impression (a surrogate for user information). For simplicity we consider the case of one bidder. Nevertheless, our analysis can be easily extended to the general case.

Let $\bar{F}(p | Q) = \mathbb{P} \{ B \ge p \mid Q\}$ be the conditional probability that the bid from AdX is greater than $p$ given that the impression quality vector is $Q$. Additionally, we define the \emph{conditional revenue function} as $r(s|Q) = s \bar{F}^{-1}(s|Q)$. The publisher can exploit the correlation between user information and bids to update his prior on AdX bids. Conditioning on the impression quality, we obtain that the maximum expected revenue under opportunity cost $c$, denoted by $R(c|Q)$, is now
\begin{align}\label{R(c|Q)}
	R(c|Q) &= \max_{0 \le s \le 1} r(s|Q) + (1 - s) c.
\end{align}
In order to apply the results from the previous sections we require that the conditional revenue function $r(\cdot|Q)$ is regular for all qualities $Q$ almost surely.

The rest of the analysis follows in a straightforward way by conditioning on the impression quality. The dual problem \eqref{DUAL2} now reads
\begin{align*}
	\min_{v} & \; \left\{ \mathbb{E} R\left(\max_{a \in \mathcal{A}_0} \{Q_a - v_a\} \mid Q\right) + \sum_{a \in \mathcal{A}}v_a \rho_a \right\},
\end{align*}
where we replaced the maximum expected revenue by $R(\cdot|Q)$. It is worth noting that now the optimal reserve price to be submitted to AdX depends both on the maximum contract adjusted quality, and the actual realization of the quality vector.

\section{Conclusion}

Ad Exchanges are an emerging market for the real-time sale of online ad slots on the Internet. Despite the popularity of this emerging market, many publishers are currently not jointly optimizing their inventory over AdX and their traditional reservations. Instead they first aim to fulfill their reservations and then submit their remnant inventory to the exchange. In this work, we show that there are considerable advantages for the publishers from jointly optimization over both channels. Publishers may increase their revenue streams without giving away the quality of service of their reservations contracts, which still represents a significant portion of their advertising yield. Our approach helps publishers determine when and how to access AdX to complement their contract sales of impressions. In particular, we model the publishers' problem as a stochastic control program and derive an asymptotically optimal policy with a simple structure: a bid-price control extended with a pricing function for the exchange. We also hope our insights here will help understand ad allocation problems more deeply.

Internet advertising, and in particular AdX, is likely to prove to be a fertile area of research. There are several promising directions of research stemming from this work. One intuitive approach to improve the performance of a control, which is appealing for its simplicity, consist on resolving the deterministic approximation periodically throughout the horizon. In a follow-up work we intend to show that one can indeed improve on the static control and obtain sharper bounds by resolving the DAP. Another problem that needs further study is that of learning in the case of unknown distributions, which is of great importance given the fast-paced and changing nature of the Internet. There exists independent research on online algorithms for capacity allocation and online pricing for repeated auctions, but none on the joint optimization problem. Finally, as more publishers reach out for AdX, advertisers will have the opportunity to buy their inventory from either market. The existence of two competing channels, the exchange as a spot market and the reservations as future market, introduces several interesting research questions. For example, how should publishers price their contracts and allocate their inventory, and how should advertisers hedge their campaign between these two markets. We hope that this work pave the way for further research on this important topic.

\ACKNOWLEDGMENT{%
We thank Omar Besbes, Ciamac Moallemi, Daryl Pregibon, Ana Radovanovic, Nemo Semret, and Nicol\'as Stier for helpful discussions.
}


\vspace{2em}

\begin{APPENDICES}

\section{Proofs of Statements}

\small

\begin{proof}{Proof of Proposition~\ref{thm:revenue-regularity}.}
First, observe that for all $c$ the objective function of \eqref{R(c)} is concave and continuous in $s$, and the feasible set is compact. Hence, by Weierstrass Theorem the set of optimal solutions is non-empty and compact. Thus, both $R(c)$ and $s^*(c)$ are well-defined.

Second, $R(c) \ge c$ follows from letting $s=0$. To see that $R(c)$ is non-increasing, let $c < c^\prime$, and $s^*$ be the optimal solution under cost $c$. Then, $R(c) = r(s^*) + (1- s^*) c \le r(s^*) + (1- s^*) c^\prime \le R(c^\prime)$ where the first inequality follows because $s^* \le 1$, and the second because no solution is better than the optimal. To see that $R(c) - c$ is non-increasing, let $c^\prime < c$, and $s^*$ be the optimal solution under cost $c$. Then, by a similar argument we get that $R(c^\prime) - c^\prime \ge r(s^*) - s^* c^\prime \ge r(s^*) - s^* c = R(c) - c$. Convexity follows in a similar way (this is a standard result).

Third, observe that the objective function of \eqref{R(c)} is jointly continuous in $s$ and $c$. Thus, by the Maximum Theorem $R(c)$ is continuous in $c$, and $s^*(c)$ is upper-hemicontinuous.

Finally, because $r(s) + (1-s) c$ has decreasing differences in $(s,c)$ and the feasible set is a lattice, by Topkis's Theorem $s^*(c)$ is non-increasing in $c$. The result for $p^*(c)$ follows from the fact that $\bar{F}^{-1}(s)$ is non-increasing in $s$.\Halmos
\end{proof}

\vspace{1em}

\begin{proof}{Proof of Theorem~\ref{thm:DAPcontrol}.}
The optimality conditions of $v$ for problem \eqref{DUAL2} imply that the directional derivative of $\psi(v)$ along any direction is greater or equal to zero. In particular for each advertiser $a \in \mathcal{A}$ it should the case that $\nabla_{\mathbf{1}_a} \psi(v) \ge 0$, and $\nabla_{-\mathbf{1}_a} \psi(v) \ge 0$. Applying proposition \ref{prop:gradient} to both directions, together with the fact that there is zero probability of a tie occurring, we get that
\begin{align*}
	\mathbb{E} \left[ \left(1 - s^* \left(Q_a - v_a \right) \right) \mathbf{1} \{ Q_a - v_a > Q_{a^\prime} - v_{a^\prime} \: \forall a^\prime \in \mathcal{A}_0  \} \right] = \rho_a,
\end{align*}
and the result follows.\Halmos
\end{proof}

\vspace{1em}

\begin{proof}{Proof of Proposition \ref{homocontrols}.}
Let $\vec{s} = \{s_n(\cdot)\}_{n=1,\ldots,N}$ and $\vec{\imath}=\{i_n(\cdot)\}_{n=1,\ldots,N}$ be any feasible vectors of controls. Let $\bar{s}$ be the mean of the controls (in terms of prices $\bar{p}$ would be the generalized $\bar{F}$-mean), which is defined point-wise $\bar{s}\Q = \frac 1 N \sum_{n=1}^N s_n\Q$. Similarly, let $\bar{\imath}$ be such that $\bar{\imath}_a\Q = \frac 1 N \sum_{n=1}^N i_{n,a}\Q$ point-wise for all $a \in \mathcal{A}$. We will show that solution in which $(\bar{s},\bar{\imath})$ are used for all impressions is a feasible control with greater or equal revenue than the original one.

First, for the feasibility of $(\bar{s},\bar{\imath})$ observe that for each advertiser $a \in \mathcal{A}$
\begin{align*}
	C_a &= 
	\mathbb{E}_{Q}\left[ \sum_{n=1}^N  i_{n,a}\Q \right] = N \mathbb{E}_{Q}\left[ \bar{\imath}_a\Q \right],
\end{align*}
where the first equation follows from the feasibility of $(\vec{p},\vec{\imath})$ and the linearity of expectation, and the second from substituting $\bar{\imath}_a$ pointwise for $Q$. Clearly, from the convexity of $\mathcal{P}$ it follows that $(\bar{s},\bar{\imath}) \in \mathcal{P}$.

Second, we denote by $J^D(\vec{s},\vec{\imath})$ and $J^D(\bar{s},\bar{\imath})$ the objective value of the solutions $(\vec{s},\vec{\imath})$ and $(\vec{s},\vec{\imath})$ respectively. We have that
\begin{align*}
 	J^D(\vec{s},\vec{\imath}) & = \mathbb{E} \left[ \sum_{n=1}^N r \left( s_n\Q \right) + \sum_{a \in \mathcal{A}} Q_a \sum_{n=1}^N i_{n,a}\Q \right] 
    \le \mathbb{E} \left[  N r \left( \frac 1 N \sum_{n=1}^N s_n\Q \right) + \sum_{a \in \mathcal{A}} Q_a \sum_{n=1}^N i_{n,a}\Q \right] \\
	& = N \mathbb{E} \left[ r \left( \bar{s}\Q \right) + \sum_{a \in \mathcal{A}} Q_a \bar{\imath}_a\Q \right] = J^D(\bar{s},\bar{\imath}),
\end{align*}
where the inequality follows from the concavity of the revenue function, and the second equality from substituting $\bar{s}$ and $\bar{\imath}$ pointwise for all $Q$.\Halmos
\end{proof}

\vspace{1em}


Given a subset of the quality space $D \subseteq \Omega$, we define the measure $\mathbb{P}_R (D)$ as the probability that the quality vector belongs to that subset and the impression is rejected by the Ad Exchange when the optimal survival probability is used. More formally,
\begin{align*}
	\mathbb{P}_R (D) &= \mathbb{E} \left[ \left(1 - s^* \left(\max_{a \in \mathcal{A}_0} \{Q_a - v_a\} \right) \right) \mathbf{1} \{ Q \in D \} \right].
\end{align*}
Notice that the latter is not a probability measure since $\mathbb{P}_R (\Omega) \le 1$. Proposition \ref{prop:gradient} characterizes the directional derivative of the objective function of the dual along some directions that, as we will show later, are of particular interest. Results are given in terms of the measure $\mathbb{P}_R$.

\begin{prop}\label{prop:gradient} Given a subset $\alpha \in \mathcal{A}$, the directional derivative of the objective function of the dual w.r.t. directions $\mathbf{1}_\alpha$ and $-\mathbf{1}_\alpha$ are respectively
\begin{align*}
	\nabla_{\mathbf{1}_\alpha} \psi(v) &= -\mathbb{P}_R \left\{ \max_{a \in \alpha} \{Q_a - v_a\} > \max_{a \in \mathcal{A}_0 \setminus \alpha} \{Q_a - v_a \}  \right\} + \sum_{a \in \alpha} \rho_a, \\
	\nabla_{-\mathbf{1}_\alpha} \psi(v) &= \mathbb{P}_R \left\{ \max_{a \in \alpha} \{Q_a - v_a\} \ge \max_{a \in \mathcal{A}_0 \setminus \alpha} \{Q_a - v_a \}  \right\} - \sum_{a \in \alpha} \rho_a.
\end{align*}
\end{prop}
\begin{proof}{Proof of Proposition \ref{prop:gradient}.}

We consider first the direction $\mathbf{1}_\alpha$. Notice that the random function $R\left(\max_{a \in \mathcal{A}_0} \{Q_a - v_a \}\right)$ is convex, and thus directionally differentiable. We first show that $\psi(v)$ is finite. From Assumption \ref{regular} we have that the revenue function is bounded by $r(s) \le M$, and thus $R(c) \le M + \max(c,0) \le M + |c|$. Therefore, using the triangle inequality we obtain that
$$
    \phi(v) \le M + \mathbb E |\max_{a \in \mathcal{A}_0} \{Q_a - v_a \}| \le M + \sum_{a \in \mathcal{A}} \mathbb E |Q_a| + |v_a| < \infty
$$
We can now apply Theorem 7.46 in \citet{Shapiro2009} and obtain that $\psi(v)$ is directionally differentiable at $v$ and that one can exchange expectation and directional derivative. Putting all together we get that
\begin{align*}
\nabla_{\mathbf{1}_\alpha} \psi(v) &= \mathbb{E} \left[ \nabla_{\mathbf{1}_\alpha} R\left(\max_{a \in \mathcal{A}_0} \{Q_a - v_a \}\right) \right] + \sum_{a \in \alpha} \rho_a \\
&= \mathbb{E} \left[R^\prime \left(\max_{a \in \mathcal{A}_0} \{Q_a - v_a \}\right) \nabla_{\mathbf{1}_\alpha} \left\{ \max_{a \in \mathcal{A}_0} \{Q_a - v_a \} \right\} \right] + \sum_{a \in \alpha} \rho_a,
\end{align*}
where the second equation follows from the chain rule. We conclude by the fact that $R^\prime(c) = 1 - s^*(c)$ and $\nabla_{\mathbf{1}_\alpha} \left\{ \max_{a \in \mathcal{A}_0} \{Q_a - v_a \} \right\} = - \mathbf 1 \left\{ \max_{a \in \alpha} \{Q_a - v_a\} > \max_{a \in \mathcal{A}_0 \setminus \alpha} \{Q_a - v_a \}  \right\}$. A similar result follows for the opposite direction $-\mathbf{1}_\alpha$ from the fact that $\nabla_{-\mathbf{1}_\alpha} \left\{ \max_{a \in \mathcal{A}_0} \{Q_a - v_a \} \right\} = \mathbf 1 \left\{ \max_{a \in \alpha} \{Q_a - v_a\} \ge \max_{a \in \mathcal{A}_0 \setminus \alpha} \{Q_a - v_a \}  \right\}$.\Halmos

\end{proof}

\vspace{1em}

\begin{proof}{Proof of Proposition \ref{prop:flow-feasible}.}
The proof proceeds by contradiction, that is, we assume that there is no feasible flow. First, we cast the feasible flow problem as a maximum flow problem. Feasibility would imply the existence of a flow with value $1 - \mathbb{P}(\emptyset\text{-tie})$. But since we assume that no such feasible flow exists, by the max-flow min-cut theorem there should exists a cut with value strictly less than $1 - \mathbb{P}(\emptyset\text{-tie})$. The contradiction arises because the optimality conditions of $v$ for the dual problem \eqref{DUAL2} imply that the every cut is lower bounded by $1 - \mathbb{P}(\emptyset\text{-tie})$.

In order to write the feasible flow problem as a maximum flow problem, we first add a source $s$ and a sink $t$. Second, we add one arc from $s$ to each node associated to a non-empty subset $S \subseteq \mathcal{A}_0$ (left-hand side nodes) with capacity $\mathbb{P}(S\text{-tie})$. Third, we add one arc from each advertiser $a \in \mathcal{A}_0$ (right-hand side nodes) to $t$ with capacity $\rho_a$. Lastly, we set the capacity of arcs from $S$ to $a \in S$ to infinity.

Now, since no feasible flow exists, by the max-flow min-cut theorem there should be a cut with value strictly less than $1 - \mathbb{P}(\emptyset\text{-tie})$. Let $\alpha \subseteq \mathcal{A}_0$ be the advertiser nodes (right-hand) belonging to the $t$ side of a \emph{minimum} cut. Figure \ref{fig:min-cut} shows the minimum cut. Next we argue that subset nodes in the $s$ side verify that $S \cap \alpha = \emptyset$, while those in the $t$ side verify that $S \cap \alpha \neq \emptyset$. First, because the cut has minimum value, there is no arc from a subset node to an advertiser node crossing the cut (those arcs have infinity capacity). Equivalently, within the $s$ side of the cut, all subsets nodes $S \subseteq \mathcal{A}_0$ should verify that $S \cap \alpha = \emptyset$. Second, observe that any subset node with $S \cap \alpha = \emptyset$ in the $t$ side of the cut could be moved to the $s$ side of the cut without increasing the value of the cut. Hence, with no loss of generality we can assume that all subset nodes in the $t$ side of the cut verify that $S \cap \alpha \neq \emptyset$.

As a consequence, the only arcs crossing the cut are those from the source to the subsets $S \cap \alpha \neq \emptyset$, and those from advertisers $\mathcal{A}_0 \setminus \alpha$ to the sink. The value of this cut is
\begin{align*}
	\sum_{S \subseteq \mathcal{A}_0 : S \cap \alpha \neq \emptyset} \mathbb{P}(S\text{-tie}) + \sum_{a \in \mathcal{A}_0 \setminus \alpha} \rho_a.
\end{align*}
Because the value is strictly less than $1 - \mathbb{P}(\emptyset\text{-tie})$ we get that
\begin{align} \label{min-cut-value}
	\sum_{S \subseteq \mathcal{A}_0 : S \cap \alpha \neq \emptyset} \mathbb{P}(S\text{-tie}) < \sum_{a \in \alpha} \rho_a,
\end{align}
where we used that $\sum_{a \in \mathcal{A}} \rho_a + \rho_0^\text{eff} = 1 -  \mathbb{P}(\emptyset\text{-tie})$.

\begin{figure}[t]
\centering
\includegraphics{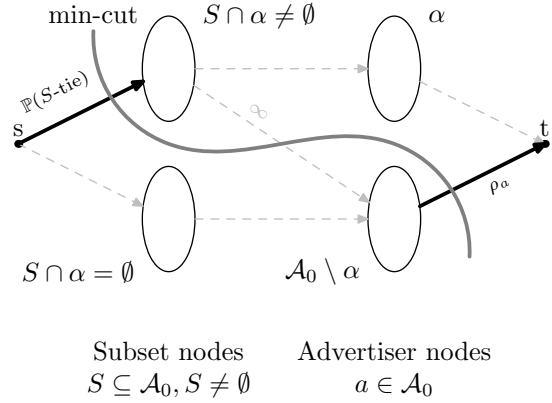}
\caption{The flow problem in the bipartite graph with a minimum cut. A source $s$ connected to the subset nodes and a sink $t$ connected to the advertisers nodes was included. $\alpha \subseteq \mathcal{A}_0$ is the subset of advertiser nodes (right-hand) belonging to the $t$ side. Note that no there is no arc from a subset node to an advertiser node crossing the cut.}\label{fig:min-cut}
\end{figure}

Next, we look at the optimality conditions of $v$ for the dual problem \eqref{DUAL2}. We distinguish between the case that $0 \notin \alpha$ and $0 \in \alpha$. First suppose that $0 \notin \alpha$, and consider the direction $-\mathbf{1}_\alpha$ that has a $-1$ if $a\in \alpha$ and $0$ elsewhere. According to proposition \ref{prop:gradient} the directional derivative of the objective at $v$ is
\begin{align*}
	\nabla_{-\mathbf{1}_\alpha} \psi(v) &= \mathbb{P}_R \left\{ \max_{a \in \alpha} \{Q_a - v_a\} \ge \max_{a \in \mathcal{A}_0 \setminus \alpha} \{Q_a - v_a \}  \right\} - \sum_{a \in \alpha} \rho_a \\
	& = \sum_{S \subseteq \mathcal{A}_0 : S \cap \alpha \neq \emptyset} \mathbb{P}(S\text{-tie})  - \sum_{a \in \alpha} \rho_a,
\end{align*}
where we have written the event that the maximum is verified non-exclusively by some advertiser $a \in \alpha$ as all $S$-ties in which some advertiser $a \in \alpha$ is involved. The optimality of $v$ implies that the directional derivative along that direction is greater or equal to zero, contradicting equation \eqref{min-cut-value}.

When $0 \in \alpha$ we consider the direction $\mathbf{1}_{\mathcal{A} \setminus \alpha}$ that has a $1$ if $a \notin \alpha$ and $0$ elsewhere. The direction derivative is now
\begin{align*}
	\nabla_{\mathbf{1}_{\mathcal{A} \setminus \alpha}} \psi(v) &= -\mathbb{P}_R \left\{ \max_{a \in \mathcal{A} \setminus \alpha} \{Q_a - v_a\} > \max_{a \in \alpha \cup \{0\}} \{Q_a - v_a \}  \right\} + \sum_{a \in \mathcal{A} \setminus \alpha} \rho_a, \\
	& = - \sum_{S \subseteq \mathcal{A}_0 : S \subseteq \mathcal{A} \setminus \alpha} \mathbb{P}(S\text{-tie}) + \sum_{a \in \mathcal{A} \setminus \alpha} \rho_a
	= \sum_{S \subseteq \mathcal{A}_0 : S \cap \alpha \neq \emptyset} \mathbb{P}(S\text{-tie})  - \sum_{a \in \alpha} \rho_a,
\end{align*}
where in the second equation we have written the event that the maximum is verified exclusively by some advertiser $a \in \alpha$ as all $S$-ties in which only advertisers in $\alpha$ are involved. Again, the optimality of $v$ implies that the directional derivative along that direction is greater or equal to zero, contradicting equation \eqref{min-cut-value}.\Halmos
\end{proof}

\vspace{1em}

\begin{proof}{Proof of Proposition \ref{thm:upperbound}.}
Let $\mu^*$ be the optimal policy for the stochastic control problem. Let $\hat{s} = \{\hat{s}_n(\cdot)\}_{n=1,\ldots,N}$ and $\hat{\imath} = \{\hat{\imath}_n(\cdot)\}_{n=1,\ldots,N}$ be deterministic vectors of controls defined as
\begin{align*}
	\hat{s}_n(Q) &=  \mathbb{E}_{\mathcal{F}_n} \left[s_n^{\mu^*}(Q) \mid Q \right]	\qquad \forall Q \: \text{pointwise},\\
	\hat{\imath}_{n,a}(Q) &=  \mathbb{E}_{\mathcal{F}_n} \left[(1-s_n^{\mu^*}(Q)) {I}_{n,a}^{\mu^*}(Q) \mid Q \right]	 \qquad \forall Q \: \text{pointwise}, a \in \mathcal{A},
\end{align*}
where the expectation is taken over the history of the system until $n$, which is denoted by $\mathcal{F}_n$, and conditional on a particular realization of $Q$. The resulting controls are independent of the history, and dependent only on the realization of $Q$ and the impression number $n$. Thus, they fulfills the first approximation and they are valid deterministic vectors of controls. We will show that $(\hat{s},\hat{\imath})$ is feasible for the DAP, and that its objective value (in the DAP) dominates the optimal objective value of the SCP. Then, we may conclude that $J_N^* \le J_N^D(\hat{s},\hat{\imath}) \le J_N^D$, because no feasible solution is better than the optimal.

First, for the contract fulfillment constraint we have that for each advertiser $a \in \mathcal{A}$
\begin{align*}
	C_a &= \mathbb{E} \left[\sum_{n=1}^N (1-X_n(s_n^{\mu^*}(Q_n))) I_{n,a}^{\mu^*}(Q_n) \right]\\ 
	& = \sum_{n=1}^N \mathbb{E} \left[ \mathbb{E}_{\mathcal{F}_n} \left[(1-s_n^{\mu^*}(Q_n)) {I}_{n,a}^{\mu^*}(Q_n) \mid Q_n \right] \right] = \sum_{n=1}^N \mathbb{E} \left[ \hat{\imath}_{n,a}(Q) \right],
\end{align*}
where the first equality follows from taking expectations to the almost sure contract fulfillment constraint of $\mu^*$, the second from the tower rule, and the third from substituting $\hat{s}$ and $\hat{\imath}$ pointwise for all $Q$ and the fact that impressions are i.i.d. Non-negativity of the controls follows trivially. Additionally, is it not hard to show that $\sum_{a \in \mathcal{A}} \hat{\imath}_{n,a}(\cdot) + s_n(\cdot) \le 1$ for all $n$. Thus, $(\hat{s},\hat{\imath})$ is a feasible deterministic control.

Second, the objective value of the optimal stochastic control is bounded by
\begin{align*}
	J_N^* 
	& = \mathbb{E} \left[ \sum_{n=1}^N \mathbb{E}_{\mathcal{F}_n} \left[ r \left(s_n^{\mu^*}(Q_n)\right) \mid Q_n \right] +
	 \sum_{a \in \mathcal{A}} Q_{n,a} \mathbb{E}_{\mathcal{F}_n} \left[(1-s_n^{\mu^*}(Q_n)) I_{n,a}^{\mu^*}(Q_n)  \mid Q_n \right] \right] \\
	& \le \mathbb{E} \left[ \sum_{n=1}^N r \left( \hat{s}_n(Q_n) \right) +
	 \sum_{a \in \mathcal{A}} Q_{n,a} \hat{\imath}_{n,a}(Q) \right] = J_N^D(\hat{s},\hat{\imath}),
\end{align*}
where the first equality follows from the tower rule and because $Q_n$ is measurable w.r.t. the conditional expectation, and the inequality from applying Jensen's inequality to the concave revenue function.\Halmos
\end{proof}

\vspace{1em}

\begin{proof}{Proof of Proposition~\ref{thm:DAPcontrol-noadx}.}
Because $B=0$, then it is not hard to show that $\bar{F}^{-1}(s) = 0$, and that the revenue function is $r(s)=0$. Hence, the revenue function is regular and satisfies Assumption \ref{regular}. Moreover, the optimal survival probability is $s^*(c) = 0$, and $R(c) = c$. The result follows from substituting these functions in Theorem \ref{thm:DAPcontrol}.\Halmos
\end{proof}

\end{APPENDICES}

\bibliographystyle{ormsv080}
\bibliography{onlinematching,../../../../common/balseiro}

\ECSwitch

\ECDisclaimer



\section{Comparison to the Primal-Dual Method}
\label{sec:withoutAdX}

Consider the allocation problem faced by a publisher in display advertising in which arriving impressions need to be assigned to advertisers, and there is no option of sending to an exchange.  This problem is a particular case of our model where the winning bid random variable is identically zero, i.e. $B=0$. The following proposition shows that the optimal controls admit simple analytical expressions.

\begin{prop}\label{thm:DAPcontrol-noadx} Suppose that ties have zero probability. Then, in the case without AdX the optimal controls are $I_a(Q)=\mathbf{1} \left\{Q_a - v_a \ge Q_{a^\prime} - v_{a^\prime} \: \forall a^\prime \in \mathcal{A}_0  \right\}$ where $v = \{v_a\}_{a \in \mathcal{A}_0}$ satisfies $v_0 = 0$ and
\begin{align*}
	\mathbb{P} \left\{Q_a - v_a \ge Q_{a^\prime} - v_{a^\prime} \: \forall a^\prime \in \mathcal{A}_0  \right\} = \rho_a \qquad \forall a \in \mathcal{A}.
\end{align*}
\end{prop}

The resulting decision rule $\arg\max_a\{Q_a - v_a\}$ is identical to
the rule studied in previous work (e.g.,~\cite{DevenurHayes2009}),
where $v_a$ is an optimal dual variable resulting from solving an
assignment problem on a sample of the data, where the distribution is
unknown.  Roughly speaking, in~\cite{DevenurHayes2009} (and similarly in other
work~\cite{Feldman2010,Vee2010,AWY09}), it
is shown that as long as the sample is of size $\approx~\epsilon n$,
the overall assignment will be $\approx \epsilon$ close to the optimal
offline solution.

In our model, the parameters of the quality distribution are known, so
we do not need to use a sample.  Of course in practice, the parameters
need to be learned, and so we would need to use a sample of the data
in order to learn them; but in many settings (including online
advertising) it is reasonable to assume that we at least know the {\em form} of
the distribution (e.g., normal, exponential, Zipf), albeit not the
specific parameters (mean, variance, covariance, etc.).  The
techniques in~\cite{DevenurHayes2009} are powerful {\em because} they
don't need to assume anything about the distribution, but it is
important to ask what can be gained from knowing the {\em form} of the
distribution, which is what we do in the remainder of this section,
both analytically and experimentally.

More formally, suppose the distribution of quality is not known with
certainty, but we have at our disposal a sample of $M$ quality vectors
$\{q_m\}_{m=1}^\M$ that may be used to pin-down the distribution.
Additionally, it is known that the qualities are drawn independently
from a population with continuous density $g(x|\theta)$, where
$\theta$ is an unknown parameter to be estimated. Let $G(x|\theta)$ be
the c.d.f.\, which we assume to be strictly monotonic.  For
simplicity, we deal with the case of one advertiser with capacity to impression ratio of $\rho$.  From Proposition~\ref{thm:DAPcontrol-noadx} the optimal DAP control is the
$(1-\rho)$-quantile of $Q$, that is, $v = \bar{G}^{-1}(\rho |
\theta)$.  We compare the asymptotic efficiency of a parametric and a
non-parametric estimation of the model.

\paragraphtitle{Parametric estimation method.} Let $\hat{\theta}_\M^\MLE$ be the maximum likelihood estimator (MLE) of the unknown parameter $\theta$. That is, $\hat{\theta}_\M^\MLE$ is an optimal solution of the program $\max_\theta \sum_{m=1}^\M \log g(q_m | \theta)$. Once we have our estimator, we plug-in the estimated distribution in the dual problem, and solve for the optimal dual variable $\hat{v}_\M^\MLE$. Again, from Proposition~\ref{thm:DAPcontrol-noadx} we have that the optimal dual variable, given our maximum likelihood estimation, is given by $\hat{v}_\M^\MLE = \bar{G}^{-1} (\rho | \hat{\theta}_\M^\MLE)$.

In turn, by the invariance property of the MLE, it is the case that $\hat{v}_\M^\MLE$ is the maximum likelihood estimator of the true optimally dual variable $v$ (see, e.g., \cite{CasellaBerger2001}). As a consequence, under some regularity conditions, we have that our new estimator is consistent, asymptotically efficient, and asymptotically normal
$$
    \sqrt{M} (\hat{v}_\M^\MLE - v) \Rightarrow \mathcal{N}(0,u(\theta)),
$$
where the $u(\theta)$ is the Cram\'er-Rao lower bound on the variance of any unbiased estimator. The Cram\'er-Rao lower bound is $u(\theta) = \left( \frac {\partial \bar{G}^{-1}} {\partial \theta} (\rho | \theta) \right)^2 I(\theta)^{-1}$, where $I(\theta) = \mathbb{E} \left[ \left(\frac{\partial}{\partial\theta} \ln g(Q|\theta)\right)^2 \right]$ is the Fisher information of parameter $\theta$.

\paragraphtitle{Non-parametric estimation method.} Considering the Sample Average Approximation (SAA) of the dual problem \eqref{DUAL2} we can obtain a non-parametric estimator of the truly optimal dual variable (see Chapter 5 from~\cite{Shapiro2009} for a review of the topic). In a SAA the expected value of the stochastic program is approximated by the sample average function over the observations $\{q_m\}_{m=1}^\M$. In our case, we have that
$$
    \hat{v}_\M = \arg\min_v \frac 1 M \sum_{m=1}^M \max\{q_m - v, 0\} + \rho v.
$$
Equivalently, the previous problem can be stated as a linear program, and in this case one obtains the training-based Primal-Dual method as described in~\cite{DevenurHayes2009}. It can be shown that, under some conditions, the non-parametric estimator $\hat{v}_\M$ is consistent and asymptotically normal~\citep{Shapiro2009}. As we shall see, this estimator is not necessarily efficient. It is not hard to prove that the sample $(1-\rho)$-quantile is an optimal solution to the SAA problem. Hence, from the asymptotic distribution of the $(1-\rho)$-quantile we have that
$$
    \sqrt {m} (\hat{v}_\M - v) \Rightarrow \mathcal{N}(0,u^\prime(\theta)),
$$
where the variance is $u^\prime(\theta) = \frac {\rho (1 - \rho)} {g(v|\theta)^2}$.

\paragraphtitle{Analysis.} Both the parametric and the non-parametric estimators converge, as the number of samples increases, to the true optimal solution. However, the non-parametric estimator is not as efficient as the parametric counterpart. Indeed, this is expected since the maximum likelihood estimator is known to be asymptotically efficient. We measure the \emph{relative efficiency} as the ratio of asymptotic variance of the non-parametric estimator to parametric one, i.e.\, $\varepsilon(\theta) = \frac {u^\prime(\theta)} {u(\theta)}$.

Until now, our analysis has been in terms of the optimal dual solution. The rationale is that the closer the dual variable is to the true value $v$, the better the performance of the policy should be. Next, we quantify analytically how does a deviation from the optimal solution impacts the performance of the policy. To assess the performance of the policy we look at the fluid limit as described in Section \ref{sec:fluid-limit}. The next proposition shows that the relative efficiency in terms of the performance is exactly equal to $\varepsilon(\theta)$. Hence, there is no loss in looking at the relative efficiency of the estimators instead.

\begin{prop} \label{thm:parametric}
The relative efficiency of the non-parametric estimator is
\begin{align}\label{releff}
    \varepsilon(\theta) = \rho (1 - \rho) I(\theta) \left( \frac {\partial \bar{G}} {\partial \theta} (v | \theta) \right)^{-2} \ge 1,
\end{align}
which is exactly equal to the relative efficiency in terms of the policies' performance.
\end{prop}

\begin{proof}{Proof of Proposition~\ref{thm:parametric}.}
For \eqref{releff}, we use that $\frac {\partial \bar{G}^{-1}} {\partial \theta} (\rho | \theta) =  \frac {\partial \bar{G}} {\partial \theta} (v | \theta) / g(v|\theta)$, which follows from the implicit function theorem. In view of Cram\'er-Rao lower bound, we have that $\varepsilon(\theta) \ge 1$.

Next, we look at the average yield of the policy as the number of impressions grows to infinity when a bid price of $u$ is employed, denoted by $\bar{J}(u)$. The limiting performance is given by
\begin{align*}
    \bar{J}(u) = \begin{cases}
        \rho \mathbb E_\theta [ Q | Q \ge u], & \text{if } u < v,\\
        \mathbb E_\theta [ Q ] - (1-\rho) \mathbb E_\theta [ Q | Q \le u], & \text{if } u \ge v.\\
    \end{cases}
\end{align*}
Under our assumptions, the performance function is continuous $u$. One would be tempted to apply the Delta Method to derive the asymptotic distribution of the performance. Unfortunately, $\bar{J}(\cdot)$ is not differentiable at $v$. However, it is the case that the performance function is semi-differentiable at $v$ with finite right-derivative $\bar{J}_+^\prime(v)\ge0$ and left-derivative $\bar{J}_-^\prime(v)\le0$. Thus, we can apply an extension of the Delta Method for directionally differentiable functions proved by \cite{Shapiro1991}, and obtain
\begin{align*}
    \sqrt {m} (\bar{J}(\hat{v}_\M) - \bar{J}(v)) \Rightarrow d\bar{J}\Big(v; \mathcal{N}(0,u^\prime(\theta))\Big), \\
    \sqrt {m} (\bar{J}(\hat{v}_\M^\MLE) - \bar{J}(v)) \Rightarrow d\bar{J}\Big(v; \mathcal{N}(0,u(\theta))\Big),
\end{align*}
where $d\bar{J}(v; \xi)$ is the G\^ateaux derivative of $\bar{J}$ at the point $v$ along the direction $\xi$, which is given by $d\bar{J}(v; \xi) = \bar{J}_+^\prime(v) \xi$ when $\xi \ge 0$, and $d\bar{J}(v; \xi) = \bar{J}_-^\prime(v) \xi$ when $\xi < 0$. Note that the asymptotic variance of the performances are $u^\prime(\theta) \cdot K$, and $u(\theta) \cdot K$ respectively, where the performance scale factor is given by $K = \frac 1 2 (\bar{J}_+^\prime(v)^2 + \bar{J}_-^\prime(v)^2) - \frac 1 {2 \pi} (\bar{J}_+^\prime(v) - \bar{J}_-^\prime(v))^2$. Thus, the relative efficiency of the performance is identical to $\varepsilon(\theta)$.\Halmos
\end{proof}

\paragraphtitle{Examples.} To fix ideas we consider two simple examples. First, suppose that $Q \sim \exp(\theta)$. The maximum likelihood estimator is given by $\hat{\theta}_\M^\MLE = \left( \frac 1 M \sum_{m=1}^M q_m \right)^{-1}$, and the Fisher information is $I(\theta) = \theta^{-2}$. The optimal dual variable is $v = - \theta^{-1} \ln \rho$. Hence, the relative efficiency is $\varepsilon(\theta) = \lfrac {1 - \rho} {\rho \ln^2\rho}$. In this case, the relative efficiency is lower bounded by $\varepsilon(\theta) \ge 1.544$. The lower bound is tight, and attained at $\rho \approx 0.2032$. The relative efficiency as a function of the capacity to impression ratio is plotted in Figure~\ref{fig:rel-eff}. As shown in the figure, the relative efficiency may be arbitrarily bad as the capacity to impression ratio gets close to zero or one.

For the next example we assume that qualities are normal with known variance $\sigma^2$ and unknown mean, that is, $Q \sim \mathcal{N}(\theta,\sigma^2)$. The maximum likelihood estimator is the sample mean, $\hat{\theta}_\M^\MLE = \frac 1 M \sum_{m=M}^m q_m$, and the Fisher Information is $I(\theta) = \sigma^{-2}$. In this case the relative efficiency is given by $\varepsilon(\theta) = 2 \pi \rho (1 - \rho) \exp \left( \Phi^{-1}(1-\rho)^2 \right)$, with $\Phi^{-1}$ being the inverse of the standard normal c.d.f.  Here the relative efficiency is lower bounded by $\varepsilon(\theta) \ge \pi/2$, with the minimum attained at $\rho = 1/2$. Interestingly, the relative efficiency is invariant under monotonic transformations of any random variable. Hence, the previous result holds too for the log-normal distribution.

\begin{figure}[t]
\centering
\includegraphics[width=8cm]{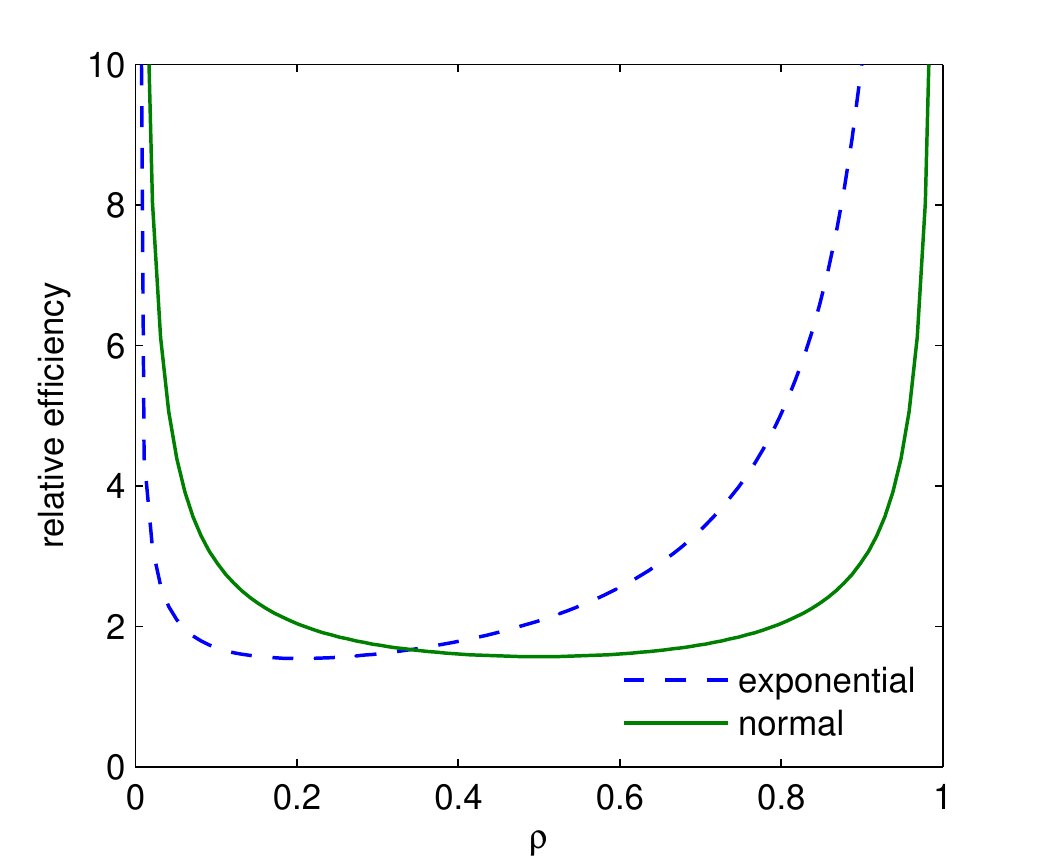}
\caption{Relative efficiency as a function of the capacity to impression ratio $\rho$ for the exponential distribution, and the normal distribution with known variance.}\label{fig:rel-eff}
\end{figure}

\paragraphtitle{Experiments.}
The previous analysis was in terms of a single advertiser;  in Section~\ref{sec:PD-compare} we show experimentally that the advantage of parametric estimation extends to multiple advertisers as well.

\section{Incorrect Assignments in the User Type Model}\label{sec:user-types}

In Section~\ref{sec:data-model} we introduced a user-type model with good-will penalties to accommodate the fact that advertisers have specific targeting criteria. If the contracts are feasible, that is, there is enough inventory to satisfy the targeting criteria; one would expect our policy to assign only impressions within the criteria. In this section we formalize the concept of a feasible operation, and give sufficient conditions under which the stochastic control policy does not assign any impressions outside of the targeting criteria.

It is straightforward to state the problem of determining whether contracts can be satisfied or not, as a feasible flow problem on a bipartite graph. The problem can be formulated on an graph with one node for each user type $T$ with a supply of $\pi(T)$, on the left side; and one node for each advertisers $a\in\mathcal{A}_0$ with a demand $\rho_a$, on the right side. Then, we say that the operation is \emph{feasible} if the user type-advertiser graph admits a feasible flow.

The feasibility of the operation, albeit necessary, does not suffice to guarantee that no impressions outside the targeting criteria are assigned to the advertisers. When advertisers compete for the same type, and one of them obtains a potentially unbounded reward for that type; it may be optimal to allow the latter advertiser to cannibalize the user type, and force the others advertisers to take types outside of their criteria. This may occur, surprisingly, for all conceivable penalties. However, if qualities are bounded, and penalties are set high enough, then the optimal policy would not recommend the assignment of impressions outside the targeting criteria. Even in this case some impressions may be incorrectly assigned in the left-over regime, but the probability of this event decays exponentially fast. We formalize this discussion in the following proposition.

\begin{prop}\label{prop:no-incorrect-assign} Suppose that the user type-advertiser graph admits a feasible flow, and that qualities and bids from AdX are bounded by $ \frac 1 A \min_a \tau_a$. Then, the stochastic control policy does not assign any impressions outside of the targeting criteria, except perhaps for the left-over regime.
\end{prop}

\begin{proof}{Proof Sketch of Proposition \ref{prop:no-incorrect-assign}.}
For simplicity we consider the case without AdX. Let $i$ be an optimal solution to the DAP, and suppose that some advertisers are assigned some types outside of their targeting criteria. We will construct another solution with greater or equal yield in which no incorrect assignments are made.

An optimal solution to the DAP is a vector of functions $i_{T,a} : \Omega_T \rightarrow [0,1]$ for $a \in \mathcal{A}_0$ and $T \in \mathcal{T}$ such that
\begin{align*}
    \sum_{T \in \mathcal{T}} \mathbb E i_{T,a} (Q) = \rho_a, \quad \forall a \in \mathcal{A}_0,\\
    \sum_{a \in \mathcal{A}_0} i_{T,a} (Q) = \pi(T), \quad \text{(a.s.)} \; \forall T \in \mathcal{T} ,
\end{align*}
We refer to each of the functions in a solution as components.

Next, we construct a feasible solution $i^0$ to the DAP from a feasible flow of the user type-advertiser graph. Take the difference $\Delta i = i^0 - i$, which is a circulation in the user type-advertiser graph. The circulation $\Delta$ may have components of mixed signs. 
Because $i^0$ has no incorrect assignments, if advertiser $a$ is assigned a type $T \not\ni a$ not in her criteria, then the circulation verifies that $\Delta i_{T,a}(Q) = - i_{T,a}(Q)$. Hence, the components with incorrect assignments are negative.

Let $a$ be an advertiser that is assigned a type $T \not\ni a$ not in her criteria, that is, $\mathbb E [i_{T,a}(Q)] > 0$. We may find an augmenting cycle $w$ containing the incorrect assignment, such that if we push some flow along this cycle, we construct another solution $i+w$ with fewer incorrect assignments. The cycle $w$ has at most $A+1$ positive components, and at most $A+1$ negative components. The cost associated to the negative components is at most $A \frac 1 A \min_{a^\prime} \tau_{a^\prime} - \tau_a \le 0$. All positive components arcs have a cost of at least zero, and the total cost of this cycle is positive. Thus, the new solution $i + w$ has greater or equal yield. Moreover, $\mathbb E [(i+w)_{T,a}(Q)] < \mathbb E [i_{T,a}(Q)]$, and no new incorrect assignments are introduced. Repeating this procedure, we may construct a solution with no incorrect assignments.\Halmos
\end{proof}

\begin{figure}[t]
\centering
    \includegraphics{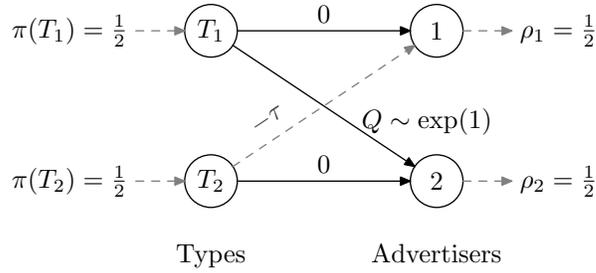}
\caption{Example with two user types, and two advertisers.}\label{fig:unbounded-example}
\end{figure}

Note that since the left-over regime is vanishingly small in proportion to the length of the horizon
(Cor.~\ref{thm:left-over}) this implies that the number of unassigned
impressions is small.  Thus in practice, a publisher may set $C_a' =
C_a + \epsilon$, discard any impressions assigned by the policy
outside the targeting criteria, and ensure that contracts are filled
properly.

Next, we prove by example that the requirement that qualities are bounded is necessary for the previous result to hold. Consider a publisher who contracts with two advertisers, and agrees to deliver one half of the arriving impressions to each one of them. Additionally, there are two impression types, denoted by $T_1$ and $T_2$, each occurring 50\% of the time. The first advertiser only cares about the first type. She obtains a reward of zero for $T_1$, and the advertisers pays a positive penalty $\tau$ each time a $T_2$ impression is assigned to her. The second advertisers admits both types, but only obtains a positive reward $Q\sim \exp(1)$ for the first type. The setup is shown in Figure~\ref{fig:unbounded-example}.

A feasible policy could assign all $T_1$ impressions to the first advertiser, and $T_2$ impressions to the second advertiser. However, such policy is not optimal. Notice that both advertisers compete for the $T_1$ impressions, and the first advertiser could extract a potentially high quality from them. It is not hard to see that the optimal dual variables are $v_1 = -\tau$, and $v_2 = 0$; and the optimal objective value is $\frac 1 2 \mathbb{E} [Q - \tau]^+ = \frac 1 2 e^{-\tau}$. Hence, it is optimal to assign those $T_1$ impressions with quality greater than $\tau$ to the second advertiser. Thus, no matter the value of the penalty, a fraction $e^{-\tau}$ of the total impression assigned to the first advertiser are undesired.

\section{Computation}\label{sec:computation}

In this section we describe show to compute the optimal policy for our data model. The main problem resides in the computation of the dual objective in \eqref{DUAL2} and its gradient given a vector of dual variables.

\paragraphtitle{Objective.} The first term of the objective can be written as
\begin{align*}
	\mathbb{E} R\left(\max_{a \in \mathcal{A}_0} \{Q_a - v_a\}\right) &=
    \sum_{\forall T} \pi(T) \mathbb{E} \left[ R\left(\max_{a \in \mathcal{A}_0} \{Q_a - v_a\} \right) \mid T \right] \\
    &= \sum_{\forall T} \pi(T) \sum_{a \in T \cup T^c } \mathbb{E} \left[ R\left(Q_a - v_a\right) \mathbf{1} \{ Q_a - v_a \ge Q_{a^\prime} - v_{a^\prime} \; \forall a^\prime \neq a \} \mid T \right] \\
    &= \sum_{\forall T} \pi(T) \left( I_{T,0}(v) + \sum_{a \in T} I_{T,a}(v) \right)
\end{align*}
where the first equation follows by conditioning on the type, and the second because the events are a partition of the sample space. Next, we show to compute the expectations $I_{T,a}(v)$.

Let $M_T(v) = \max_{a \in \mathcal{A}_0 \setminus T} \{ - \tau_a - v_a \}$ be the maximum contract adjusted quality of the advertisers (including the outside option) that are not in the type, and $\alpha_T(v)$ the set of advertisers that verify the maximum. Then, we have that
\begin{align*}
    I_{T,0}(v) &= R\left(M_T(v)\right) \mathbb{P} \{ Q_a - v_a \le M_T(v)  \; \forall a^\prime \in T \} \\
    & = R\left(M_T(v)\right) G_T (M_T(v) + v_T ),
\end{align*}
where $G_T(\cdot)$ is the c.d.f.\ of $Q_T$, and $v_T$ is the vector of dual variables for the advertisers in the type.

For $a\in T$, we compute the expectation by conditioning on the continuous random variable $Q_a$. Further, suppose that we partition the mean vector and covariance matrix in a corresponding manner. That is, $\mu_T = \left(\begin{smallmatrix} \mu_a \\ \mu_{-a} \end{smallmatrix}\right)$, and $\Sigma_T = \left(\begin{smallmatrix} \Sigma_{a,a} & \Sigma_{a,-a} \\ \Sigma_{-a,a} & \Sigma_{-a,-a}\end{smallmatrix}\right)$. For instance, $\mu_{-a}$ gives the means for the variables in $T \setminus \{a\}$, and $\Sigma_{-a,-a}$ gives variances and covariances for the same variables. The matrix $\Sigma_{-a,a}$ gives covariances between variables in $T \setminus \{a\}$ and set $a$ (as does matrix $\Sigma_{a,-a}$). Because the marginal distribution of a multivariate normal is an univariate normal, we have that $Q_a \sim \ln \mathcal{N}(\mu_{a}, \Sigma_{a,a})$. We denote by $g_{T,a}(\cdot)$ the p.d.f.\ of $Q_a$. Similarly, let $Q_{-a}$ be the vector of qualities for advertisers in $T \setminus \{a\}$. Conditioning on $Q_a = q_a$, the distribution of $Q_{-a}$ is log-normal with mean vector $\mu_{-a} - \Sigma_{-a,a} \lfrac {q_a - \mu_a} {\Sigma_{a,a}}$, and covariance matrix $\Sigma_{-a,-a} - \lfrac {\Sigma_{-a,a} \Sigma_{a,-a}} {\Sigma_{a,a}}$. We denote its c.d.f.\ by $G_{T,-a}(\cdot)$. Putting all together, we have that
\begin{align*}
    I_{T,a}(v) &= \mathbb{E} \left[ R\left(Q_a - v_a\right) \mathbb{P} \{ Q_{a^\prime} - v_{a^\prime} \le Q_a - v_a \; \forall a^\prime \neq a \mid Q_a \} \mid T \right] \\
    &= \int_{v_a + M_T(v)}^\infty R(q_a - v_a) G_{T,-a}(q_a - v_a + v_{-a}) g_{T,a}(q_a) \text{ d} q_a,
\end{align*}
where $v_{-a}$ is the vector of dual variables for advertisers in $T \setminus \{a\}$.

\paragraphtitle{Gradient.} The forward derivative of the dual objective can be written as
\begin{align*}
	\nabla_{a} \psi(v) &= -\mathbb{P}_R \left\{ Q_a - v_a > \max_{a \in \mathcal{A}_0 \setminus a} \{Q_{a^\prime} - v_a^\prime \}  \right\} + \rho_a\\
    & = - \sum_{\forall T} \pi(T) \mathbb E \left[ \left(1 - s^*(Q_a - v_a)\right) \mathbf 1 \left\{ Q_a - v_a > \max_{a \in \mathcal{A}_0 \setminus a} \{Q_{a^\prime} - v_a^\prime \}  \right\} \mid T \right] + \rho_a\\
    &= - \sum_{T: a \in T} \pi(T) P_{T,a}(v) - \sum_{\substack{T: a \not\in T \\ a \in \alpha_T(v), |\alpha_T(v)|=1}} \pi(T) P_{T,a}(v) + \rho_a,
\end{align*}
where the contributing types for the forward derivative are those where $a$ is in, and those where $a$ is not in but verifies exclusively the maximum of the types not in ($M_T(v)$). If two or more advertisers verify the maximum $M_T(v)$, then increasing $v_a$ does not have an impact of the type's contribution to the objective. When $a\not\in T$, the expectation is given by
\begin{align*}
    P_{T,a}(v) = \big(1 - s^*(M_T(v))\big) G_T (M_T(v) + v_T).
\end{align*}
Similarly to the objective, when $a\in T$ we have that
\begin{align*}
    P_{T,a}(v) &= \int_{v_a + M_T(v)}^\infty  \big(1 - s^*(q_a - v_a)\big) G_{T,-a}(q_a - v_a + v_{-a}) g_{T,a}(q_a) \text{ d} q_a.
\end{align*}

The backward derivative is computed in a similar fashion. The only exception is that, when $a \not\in T$, and $a$ verifies the maximum $M_T(v)$,the advertiser always contributes to the derivative regardless of the number of advertisers that attain the maximum. Hence,
\begin{align*}
	\nabla_{-a} \psi(v) &= \sum_{T: a \in T} \pi(T) P_{T,a}(v) + \sum_{\substack{T: a \not\in T \\ a \in \alpha_T(v)}} \pi(T) P_{T,a}(v) - \rho_a.
\end{align*}

\paragraphtitle{Optimization.} We solve the dual problem \eqref{DUAL2} using a Gradient Descent Method. At each step the objective and its objective are computed as described previously. Notice that, when multiple advertisers verify a tie, the objective is not differentiable. In this case a descent direction is constructed using the forward and backward derivatives (if possible).

\paragraphtitle{Ties.} For the following, we assume that the instance is not \emph{degenerate}, that is, the variances within the types are positive, and no two advertisers are perfectly correlated. Then, within each type, non-trivial ties can only occur between the advertisers that are not in the type (we refer to the non-trivial ties as those in which multiple advertisers attain the same contract adjusted quality). Moreover, there can be at most one tie within each type, and this happens when the maximum $M_T(v)$ is verified by many advertisers, that is $|\alpha_T(v)| > 1$. With some abuse of notation, the probability of such a tie is given by $\pi(T) P_{T,\alpha_T(v)}$ and it should be split among the advertisers $\alpha_T(v)$. Note that the number of non-trivial ties is $O(T)$, and the tie-breaking rule can be computed efficiently by solving a feasible flow problem.

\section{Fluid Limit}\label{sec:fluid-limit}

Exploiting our generative model we can construct a fluid model, and obtain the limiting performance of an arbitrary bid-price policy as the number of impressions grows to infinity. We first describe the fluid equations governing the dynamics of the fluid model, then we construct a solution to such system, and then prove that the stochastic algorithm satisfies the fluid equation in the limit.
For simplicity, we focus our analysis on the case with no AdX, and no ties; though, a similar analysis applies to the more general case.

In the following we analyze the performance of the stochastic control policy when implementing some (sub-optimal) bid-prices $v$. Let $\omega$ be a sample path, 
$J_n(\omega)$ be the cumulative yield collected up to impression $n$, and $S_{n,a}(\omega)$ the number of impressions assigned to advertiser $a$ up to time $n$. We extended the previous definitions for an arbitrary time, by taking their linear interpolations, so that they are continuous. The previous functions are random elements on $\mathbb{C}[0,\infty)$. We shall construct the fluid limit by scaling capacity and time proportionally to infinity, and considering a continuous flow of impressions arriving during an horizon of length 1. More formally, we define $\bar{S}_a(t) = \lim_{N \rightarrow \infty} N^{-1} S_{tN,a}(\omega)$, which can be interpreted as the fraction of impressions assigned to advertiser $a$ by time $t$. Similarly, we define $\bar{J}(t) = \lim_{N \rightarrow \infty} N^{-1} J_{tN}(\omega)$ as the cumulative yield up to time $t$. We are interested in computing $\bar{J}(1)$, the total limiting yield of the algorithm under bid-prices $v$.

When capacity is scaled, each advertiser has a capacity of $\rho_a$, and the fluid model should satisfy the following differential equations
\begin{subequations}\label{fluid}
\begin{align}
    {J}^\prime(t) &= \displaystyle \sum_{a \in \mathcal{A}(t)} \mathbb{E} \Big[ Q_a \mathbf{1} \big\{ a = \mathop{\arg\max}_{a^\prime \in \mathcal{A}(t)} \{Q_{a^\prime} - v_{a^\prime}\}  \big\} \Big], \label{fluid3}\\
    {S}_a^\prime(t) &= \displaystyle \mathbb{P} \Big\{ a = \mathop{\arg\max}_{a^\prime \in \mathcal{A}(t)} \{Q_{a^\prime} - v_{a^\prime}\} \Big\}, \quad \forall a \in \mathcal{A}_0 \label{fluid2}\\
    \mathcal{A}(t) &= \left\{ a\in\mathcal{A}_0 : \bar{S}_a(t) < \rho_a \right\}, \label{fluid1}
\end{align}
\end{subequations}
with the initial conditions ${S}_a(0) = 0$, and ${J}(0) = 0$. In \eqref{fluid1}, $\mathcal{A}(t)$ is the set of advertisers that are yet to be fulfilled (including advertiser 0, which is fulfilled when the time comes all impressions should be assigned directly to the advertisers), and \eqref{fluid2} determines the rate at which impressions are assigned to each advertiser. When one advertiser is fulfilled and the fraction of impressions $\bar{S}_a^\prime(t)$ reaches its capacity $\rho_a$, it is excluded from $\mathcal{A}(t)$, and its rate is driven to zero. Finally, \eqref{fluid3} determines the rate at which yield is generated.

It is not hard to see that the solution to the fluid equations \eqref{fluid} is piecewise linear, and continuous. We construct a solution as follows. Let an \emph{epoch}, denoted by $t^k$, be the time in which the contract of any advertiser is fulfilled (including advertiser 0). The horizon $[0,1]$ is partitioned in consecutive \emph{pieces}, each culminating with an epoch. The $k^{\rm th}$ piece spans the interval $[t^k,t^{k+1})$, and has a length of $\Delta^k = t^{k+1}-t^k$. Since one advertiser is fulfilled at each epoch, there are at most $A+1$ pieces.

Let $\mathcal{A}^k$ be set of advertisers yet to be satisfied at the beginning of stage $k$ as given by \eqref{fluid1}, $r_a^k$ be the service rate for advertiser $a$ during stage $k$ as given by \eqref{fluid2}, and $y^k$ be the yield rate during stage $k$ as given by \eqref{fluid3}. The length of a stage is determined by the advertiser that is first fulfilled. Once determined, the solution is constructed recursively as follows
\begin{subequations}
\begin{align}
    \Delta^k &= \min_{a \in \mathcal{A}^k} \left\{ \frac {\rho_a - S_a(t_k)} {r_a^k} \right\}, \label{fluid-alg1}\\
    t^{k+1} &= \Delta^k + t^k, \nonumber\\
    S_a(t^{k+1}) &= S_a(t^k) + r_a^k \Delta^k, \nonumber\\
    J(t^{k+1}) &= J(t^k) + y^k \Delta^k, \nonumber\\
    \mathcal{A}^{k+1} &= \mathcal{A}^k \setminus {a^k}, \nonumber
\end{align}
\end{subequations}
where $a^k$ is the advertiser that verifies the minimum in \eqref{fluid-alg1}, and initially $t^0 = 0$, and $\mathcal{A}^{0} = \mathcal{A}_0$. The functions $S_a$ and $J$ are obtained as the linear interpolation of the values at the endpoints of the interval. Fortunately, the rates can be easily obtained by evaluating the dual objective and its gradient. Let $\psi^k(v)$ be the objective in \eqref{DUAL2} when restricting to the set of advertisers $\mathcal{A}^k$. Then, we have that $r^k = \rho - \nabla \psi^k(v)$, and $y^k = \psi^k(v) - v \cdot \nabla\psi^k(v)$.

We conclude by showing that functions obtained are actually the fluid limit of the stochastic process induced by the algorithm.

\begin{prop} The fluid limits $\bar{S}_a(t)$, and $\bar{J}(t)$ are a solution to the fluid equations \eqref{fluid}.
\end{prop}
\begin{proof}{Proof.}
Consider the sequence of random elements $\bar{S}_{N,a} (t,\omega) = N^{-1} S_{tN,a}(\omega)$, and $\bar{J}_{N} (t,\omega) = N^{-1} J_{tN}(\omega)$. We would like to show that the previous sequences are tight. By Theorem 8.3 in \cite{Billingsley1968}, a sequence of random elements $\{X_N\}$ in $\mathbb{C}[0,\infty)$ is tight iff (i) $\{X_N(0)\}$ is tight, and (ii) for all $\epsilon>0$ and $\eta>0$, there is a $\delta>0$ and an integer $N_0$ such that $\mathbb{P}\{ \sup_{t\le t^\prime\le t+\delta} |X_{N}(t^\prime) - X_{N}(t)| \ge \epsilon \} \le \eta$ for $N\ge N_0$.

The first condition is trivially satisfied for both sequences. Disregarding integrality issues, which are not important in the analysis, and using the fact that the processes are non-decreasing, we have that
$$
    \sup_{t\le t^\prime\le t+\delta} |\bar{S}_{N,a}(t^\prime) - \bar{S}_{N,a}(t)|
    \le \frac 1 N \left(S_{(t+\delta)N,a} - S_{tN,a}\right) \le \delta.
$$
Thus, by picking $\delta < \epsilon$ the second condition is satisfied for the number of impressions assigned. For the yield processes, employing Markov's inequality, and the bound $\mathbb{E}R_n^2 \le A^2 \max_a \{\mathbb{E} Q_a^2\}$ for the yield in any single period, we obtain
$$
    \frac 1 \delta \mathbb{P}\left\{ J_{(t+\delta)N} - J_{tN} \ge N\epsilon \right\}
    \le \frac 1 {\delta N^2 \epsilon^2} \mathbb{E} \Big[\sum_{n=tN}^{(t+\delta)N} R_n \Big]^2
    \le \delta \frac{ A^2 \max_a \{\mathbb{E} Q_a^2\}} {\epsilon^2},
$$
which can be bounded from above by $\eta$ by picking a small enough $\delta$, and $N>1/\delta$.


Next, we show that the fluid limit converges to the solution of the equation \eqref{fluid2}. Before proceeding we state some definitions. Let the stopping time $n_N^k = \inf\{n : S_{n,a} \ge C_a \text{ for some } a \in \mathcal{A}_N^{k-1}\}$ be the time in which the contract of the $k^\text{th}$ advertisers is fulfilled. In our previous terminology, $n_N^k$ is the  $k^\text{th}$ epoch and the beginning of the $k^\text{th}$ piece. Similarly, we let $\mathcal{A}_N^{k} = \{a \in \mathcal{A}_0 : S_{n_N^k, a} < C_a\}$ to be the set of advertisers that are active during the $k^\text{th}$ piece. The initial values are given by $n_N^0 = 1$ and $\mathcal{A}_N^{0} = \mathcal{A}_0$.

We intend to show that $N^{-1} n_N^k \rightarrow t^k$, $\mathcal{A}_N^k \rightarrow \mathcal{A}^k$, and $N^{-1} S_{n_N^k,a} \rightarrow S_a(t^k)$ as $N \rightarrow \infty$ in an almost sure sense. We proceed by induction in $k$. The base case follows trivially. Suppose that our claim holds for $k$, we intend to show that it holds for $k+1$. By the definition of the $(k+1)^\text{th}$ epoch, it should be the case that $S_{n_N^{k+1} - 1,a} < C_a$ for all $a \in \mathcal{A}_N^{k}$ and $S_{n_N^{k+1},a} \ge C_a$ for exactly one advertiser $a \in \mathcal{A}_N^{k}$. Notice that the number of impressions assigned to ad $a$ up to the stopping time can be written as
\begin{align}\label{sumofa}
    \frac {S_{n_N^{k+1},a}} N =  \frac {S_{n_N^k,a}} {N} +
    \frac {n_N^{k+1} - n_N^k} {N} \frac {\sum_{n=n_N^k+1}^{n_N^{k+1}} I_{n,a}} {n_N^{k+1} - n_N^k},
\end{align}
where the summands on the right-hand size are i.i.d.~bernoulli random variables with success probability $\mathbb{P} \Big\{ a=\mathop{\arg\max}_{a^\prime \in \mathcal{A}_N^k} \{Q_{a^\prime} - v_{a^\prime}\} \Big\}$. From the induction hypothesis, together with the Triangular Strong Law of Large Numbers, we get that \eqref{sumofa} goes to $S_a(t^k) + (\lim N^{-1} n_N^{k+1} - t^k) r_a^k$ since the success probability converges to $r_a^k$. It is not hard to see that the same limit holds for $N^{-1} S_{n_N^{k+1}-1,a}$. Thus, we have that for all advertisers yet to be satisfied it should be the case that $S_a(t^k) + (\lim N^{-1} n_N^{k+1} - t^k) r_a^k \le \rho_a$, and for at least one advertiser $S_a(t^k) + (\lim N^{-1} n_N^{k+1} - t^k) r_a^k \ge \rho_a$. Combining this expressions we get that in the limit the stopping time should satisfy
$$
    \lim \frac {n_N^{k+1}} {N} = t^k + \min_{a \in \mathcal{A}^k} \left\{ \frac {\rho_a - S_a(t_k)} {r_a^k} \right\} = t^{k+1}.
$$
In turn this implies that $N^{-1} S_{n_N^{k+1},a} \rightarrow S_a(t^{k+1})$ and $\mathcal{A}_N^{k+1} \rightarrow \mathcal{A}^{k+1}$, thus concluding the inductive step.
\end{proof}

\end{document}